\documentclass[12pt,a4paper]{article}

\usepackage[top=25mm, bottom=25mm, left=25mm, right=25mm]{geometry}

\usepackage{tikz} 
\usetikzlibrary{shapes.geometric,arrows,positioning}
 \usetikzlibrary{shapes.gates.logic.US,trees,positioning,arrows}
\usepackage{afterpage}
\usepackage{paralist}
\usetikzlibrary{arrows,calc,fit} 
\usepackage{lipsum}
\usepackage{rotating}
\usepackage{mathrsfs}
\usepackage{tikz}
\usetikzlibrary{positioning,shadows,arrows}
\usepackage{enumerate}
\usepackage{ amsfonts,amsthm}
\usepackage{ multirow, graphicx,appendix}
\usepackage{amssymb}
\usepackage{amsbsy}

\usepackage{chemarr} 

\usepackage{amsmath}

\usepackage{color}
\usepackage{amsmath,amsfonts,amssymb,amscd,amsthm,xspace}

\usetikzlibrary{patterns}
\usetikzlibrary{decorations.pathmorphing, shapes}
\usetikzlibrary{automata}
\PassOptionsToPackage{usernames,dvipsnames,svqnames}{xcolor}

\usepackage{young}
\usepackage[vcentermath]{youngtab}

\title{The bubble algebras at roots of unity}
\author{ Mufida M. Hmaida}
\date{}

\begin{document}

\newcommand{\Pn}{\mathbb{P}_{n,2}(\delta_0 , \delta_1 ) }
\newcommand{\Pm}{\mathbb{P}_{n,m} (\delta_0 , \dots , \delta_{m-1} )}
\newcommand{\PN}{\mathbb{P}_{n,2} }
\newcommand{\pn}{\mathcal{P}_{n,2} }
\newcommand{\PM}{\mathbb{P}_{n,m} }
\newcommand{\ppm}{\mathcal{P}_{n,m} }

\newcommand{\Tn}{\mathbb{T}_{n,2}(\delta_0 , \delta_1 ) }
\newcommand{\TN}{\mathbb{T}_{n,2} }
\newcommand{\Tm}{\mathbb{T}_{n,m}(\delta_0 , \dots , \delta_{m-1} ) }
\newcommand{\TM}{\mathbb{T}_{n,m} }
\newcommand{\tn}{\mathcal{T}_{n,2} }
\newcommand{\tm}{\mathcal{T}_{n,m} }

\newcommand{\Sn}{\mathfrak{S}_{n,2}}
\newcommand{\Sm}{\mathfrak{S}_{n,m}}
\newcommand{\Sym}{\mathfrak{S}}

\newcommand{\TL}{\mathsf{TL}_{n}( \delta) }
\newcommand{\Tl}{\mathsf{TL}_{n} }

\newcommand{\Pp}{\mathbb{P}_{n} ( \delta ) }
\newcommand{\ie}{\mathsf{e}}

\newcommand{ \lambd}{ \underline{\lambda} }
\newcommand{ \muu }{ \underline{\mu} }

\newcommand{\Top}{\operatorname{top}}
\newcommand{\Bot}{\operatorname{bot}} 
\newcommand{\Hom}{\operatorname{Hom}}
\newcommand{\im}{\operatorname{im}}
\newcommand{\dom}{\operatorname{dom}}
\newcommand{\Par}{\operatorname{Par}}
\newcommand{\id}{\mathit{id}}
\newcommand{\Rad}{\operatorname{Rad}}
\newcommand{\rank}{\operatorname{rank}}
\newcommand{\Ker}{\operatorname{Ker}}

\newcommand{\A}{\mathbb{A}  }
\newcommand{\F}{\mathbb{C}  }
\newcommand{\Z}{\mathbb{Z}  }
\newcommand{ \Cat }{ \mathsf{ C } }  
\newcommand{ \Car }{ \mathbf{ \mathscr{ C } } }   
\newcommand{\Gr}{\mathsf{ G } }
\newcommand{\M}{\mathsf{ M } }
\newcommand{\V }{\mathsf{ V } }
\newcommand{\RV }{\mathsf{ R } }
\newcommand{\Lcell }{\mathsf{ L } }
\newcommand{ \Cc }{ \mathfrak{ C }}
\newcommand{ \I }{ \mathcal{ I }}

\theoremstyle{plain}
\newtheorem{theorem}{Theorem}[section]
\newtheorem{corollary}[theorem]{Corollary}
\newtheorem{lemma}[theorem]{Lemma}
\newtheorem{example}{Example}[theorem]
\newtheorem{proposition}[theorem]{Proposition}
\newtheorem{axiom}[theorem]{Axiom}
\theoremstyle{definition}
\newtheorem{definition}[theorem]{Definition}
\theoremstyle{remark}
\newtheorem{remark}[theorem]{Remark}

\maketitle

\textbf{Abstract.} We introduce multi-colour partition algebras $ \Pm$, which are generalization of both bubble algebras and partition algebras, then define the bubble algebra $ \Tm $ as a sub-algebra of the algebra $ \Pm $. We present general techniques to determine the structure of the bubble algebra over the complex field in the non-semisimple case.

\section{Introduction}    

 \hspace{15pt}
 In 2003, Grimm and Martin\cite{GM2} introduced a new construction, called the bubble algebra, this algebra defined entirely diagrammatically.  They investigated its generic representations and proved that it is semi-simple when none of the parameters $ \delta_i $ is a root of unity. Later,  Jegan\cite{JM} showed that the bubble algebra is a cellular algebra in the sense of Graham and Lehrer\cite{GL1}, and that it is a tower of recollement when all of the $ \delta_i $ are non-zero, as it is defined in \cite{CMPX} or \cite{MGP}. The notion of a cellular algebra was first introduced by Graham and Lehrer \cite{GL1}.  Also Jegan\cite{JM} showed how certain idempotent sub-algebra of the bubble algebra corresponded to tensor products of the Temperley-Lieb algebras and investigated the homomorphisms between the cell modules of the algebra $ \Tm $.

In this paper, we use a technique consist of reducing problems in the bubble algebra to problems in the Temperley-Lieb algebra. The representation theory of the Temperley-Lieb algebra is well known, see Martin \cite{M4}, Ridout and Saint \cite{RS1} and Westbury \cite{W1}.  All the algebras in this paper are over the complex field and all the modules are left modules.

The main results of the paper are Theorems \ref{thm5:Rad3} and \ref{thm5:general}, which determine radical series of cell modules for the algebra $ \Tm $ over the complex field and for all the tuples $ (\delta_0, \dots , \delta_{m-1} ) $ in case $ m=2 $ and $ m>2 $ respectively.

\section{Preliminaries} \label{sec1}

 \hspace{15pt}
 For $n \in \mathbb{N}$, the symbol $ \mathcal{P}_n$ denotes the set of all partitions of the set $\underline{n} \cup \underline{n'}$, where $\underline{n}=\{1, \dots , n \}$ and $\underline{n'}=\{ 1' ,  \dots , n' \}$.

Each individual set partition can be represented by a graph, the graph is drawn in a rectangle with $ n $ nodes on the top row represent the elements in the set $ \underline{n} $ and with $ n $ nodes on the bottom row of the rectangular represent the elements in the set $ \underline{n}'$, and the elements that in the same part at a partition, are represented as lines drawn connected their nodes inside the rectangular. Any diagrams are regarded as the same diagram if they representing the same partition.

Now the composition $\beta \circ \alpha $ in $\mathcal{P}_n$, where $\alpha, \beta \in \mathcal{P}_n$, is the partition obtained by placing $ \alpha $ above $ \beta $, identifying the bottom vertices of $ \alpha $ with the top vertices of $ \beta $, and ignoring any connected components that are isolated from boundaries. This product on $\mathcal{P}_n$ is associative and well-defined up to equivalence.

A $(n_1, n_2)$-partition diagram  for any $ n_1, n_2 \in \mathbb{N}^+ $ is  a diagram representing a set partition of the set  $\underline{n_1} \cup \underline{n_2'} $ in the obvious way.

  We can generalize the product on $\mathcal{P}_n$ to define a product of $(n,m)$-partition diagrams when it is defined: let $ \alpha $ be $(n_1,n_2)$-diagram and $ \beta $ be $(m_1,m_2)$-diagram, $ \beta \circ \alpha $ is defined if and only if $n_2=m_1$ and it is $(n_1,m_2)$-diagram. For example, see the following figure.
 \begin{center}
 \begin{tikzpicture}
          \draw[gray]  (-4,0.35) rectangle (-2.5,1);
           \draw (-2.75,0.35)--(-3.5,1);
           \draw (-3.75,0.35)--(-3,1);                          
          \draw [domain=0:180] plot ({0.25*cos(\x)-3.25}, {0.15*sin(\x)+0.35 });
          
          \node[above] at (-2.3 ,0.35) {\sffamily $ \circ $};
          
          \draw[gray]  (-2,0.35) rectangle (-0.7,1);
           \draw (-1.25,0.35)--(-1,1);
           \draw (-1.5,0.35)--(-1.25,1);         
           \draw [ domain=360:180] plot ({0.25*cos(\x)-1.5}, {0.15*sin(\x)+1});                    
          
\node[above] at (-0.3,0.35) {\sffamily = };
          \draw[gray]    (0,0) rectangle (1.5,0.65);
   \draw[gray]  (0,0.75) rectangle (1.5,1.4);
          \draw (0.25,0)--(0.75, 0.65 ) ;
          \draw (0.5,0.65 )--(1.25,0) ;
          \draw [domain=0:180] plot ({0.25*cos(\x)+0.825}, {0.15*sin(\x)});
           \draw (0.5,0.75)--(0.9, 1.4 );
           \draw (0.75,0.75)--(1.2, 1.4 );
           \draw [ domain=360:180] plot ({0.3*cos(\x)+0.6}, {0.15*sin(\x)+ 1.4 });  

\node[above] at (1.75,0.35) {\sffamily = };

          \draw[gray]  (2,0.35) rectangle (3.5,1);

           \draw (2.25,0.35)--(3.2,1);
           \draw (3.25,0.35)--(2.9,1);         
           \draw [ domain=360:180] plot ({0.3*cos(\x)+2.6}, {0.15*sin(\x)+1});                    
          \draw [domain=0:180] plot ({0.25*cos(\x)+2.825}, {0.15*sin(\x)+0.35 });
\end{tikzpicture}
 \end{center}

The diagrams representing partitions that spanning the Temperley-Lieb algebra $ \TL $ over (say) the complex field are planar (non-crossing) and their parts all have size two. Thus the following diagrams
\begin{center}
\begin{tikzpicture}
          \draw[gray] (0,0) rectangle (1.25,0.8);
           \draw (0.25,0) --(0.25,0.8);
           \draw (0.5,0) --(1,0.8);
            \draw[ domain=180:360] plot ({0.125*cos(\x)+0.625}, {0.3*sin(\x)+0.8});
            \draw[ domain=0:180] plot ({0.125*cos(\x)+0.875}, {0.25*sin(\x)+0});
\end{tikzpicture} 
\;\;\;\;\;\;\;
\begin{tikzpicture}
          \draw[gray] (0,0) rectangle (1.25,0.8);
            \draw[ domain=180:360] plot ({0.125*cos(\x)+0.625}, {0.2*sin(\x)+0.8});
            \draw[ domain=180:360] plot ({0.375*cos(\x)+0.625}, {0.3*sin(\x)+0.8});
            \draw[ domain=0:180] plot ({0.125*cos(\x)+0.375}, {0.25*sin(\x)+0});
            \draw[ domain=0:180] plot ({0.125*cos(\x)+0.875}, {0.25*sin(\x)+0});
\end{tikzpicture} 
\end{center}
are representing basis elements of the algebra $ \mathsf{TL}_{4}( \delta) $.

We next briefly describe  the cell modules of the algebra $ \TL $, which will be used in this paper.

A diagram representing a partition in the algebra $ \TL $ can be cut to construct a half-diagram such that all arcs on the top edge are above the cut, all arcs on the bottom edge are below the cut and each propagating line is only cut once. A half-diagram has $p$ arcs called an $(n,p)$-link state. For example, the half-diagram 
\begin{center}
\begin{tikzpicture}
          \draw (0,0.5) -- (2,0.5);
           \draw (0.75,0.1) --(0.75,0.5);
            \draw[ domain=180:360] plot ({0.125*cos(\x)+0.375}, {0.2*sin(\x)+0.5});
            \draw[ domain=180:360] plot ({0.125*cos(\x)+1.375}, {0.2*sin(\x)+0.5});
            \draw[ domain=180:360] plot ({0.375*cos(\x)+1.375}, {0.3*sin(\x)+0.5});
\end{tikzpicture}
\end{center}
 is a $ (7,3) $-link state.

As the number of propagating lines can not increase by the multiplication, we can define  left $ \Tl  $-modules $ \M_{n,p} $ which are spanned by $(n,p')$-link states with $ p' \geq p $ with action defined by putting the $ \TL $-diagram above the half-diagram then proceeds as with $ \TL $ multiplication, and finally omit any new bottom arcs. Note that 
$$ \M_{n,[^n /_2 ] } \subset \ldots \subset \M_{n,1} \subset \M_{n,0} . $$

The Temperley-Lieb algebra is a cellular algebra, with  the involution sending each diagram to its reflection in the horizontal plane, indexing set $ \{ 0, 1, \dots , [^n /_2 ] \} $ and cell modules $ \V_{n,p} := \M_{n,p} / \M_{n,p+1} $, see  \cite{GL1}. The dimension of  $  \V_{n,p}  $ is  $  \binom{n}{p} - \binom{n}{p-1} := \mathsf{d}_{n,p}$. Note that $ \binom{n}{-1}=0 $.

On each module $ \V_{n,p} $, there is a bilinear form $ \langle \; , \; \rangle_{n,p,\delta } $ defined as follows: if $ x $ and $ y $ are two $(n,p)$-link states, the scalar $ \langle x , y \rangle_{n,p,\delta } $ is computed by reflecting $x$ in a horizontal axis and identifying its vertical border with that of $y$. The value $ \langle x , y \rangle_{n,p,\delta } $ is then non-zero only if every defect (an unconnected node) of $x$ ends up being connected to one of $y$, and in this case $ \langle x , y \rangle_{n,p,\delta } = \delta^l$ where $l$ is the number of closed loops which is obtained from connecting $ x $ and $ y $. For more details see section 9.5.2 in \cite{M4}.

The matrix $ \Gr_{n,p, \delta } $ is defined to be the Gram matrix for the module $ \V_{n,p} $ that represent the form $\langle \; , \; \rangle _{n,p,\delta } $ with respect to a basis that contains all $(n,p)$-link states.

Let $ M $ be a module whose a bilinear form $ \langle \; , \; \rangle $. The radical of this form on $ M $ is the set $ \{ x \in M \mid  \langle x , y \rangle =0 \; \text{for all } \; y \in  M \}$. Define $ \RV_{n,p, \delta } $ to be the radical of the previous bilinear form on the module $ \V_{n,p} $.

As we work over a field, the radical $ \RV_{n,p} $ is a  sub-module of  $  \V_{n,p} $. If $ \delta \neq 0 $, then $ \V_{n,p} $ is cyclic and indecomposable. Moreover,  $ \Lcell_{n,p} := \V_{n,p} /  \RV_{n,p} $ is irreducible. The cell modules $ \V_{n,p} $ of the algebra $ \TL $ are irreducible except for particular values of the scalar $ \delta $. Throughout this paper, let $ \delta= q+q^{-1} $ with $ q \in \mathbb{C} $.

\begin{proposition}\cite[Section 6.4, Theorem 1]{M4}\textbf{.}  \label{thm1:TLsemisimple}
If $ q $ is not a root of unity, then the algebra $ \TL $ is semi-simple, and the modules $ \V_{n,p} $, where $ 0 \leq p \leq [^n/_2]  $, form a complete set of non-isomorphic irreducible modules of the algebra $ \TL $.
\end{proposition}

Let $ q $ be a root of unity and let $ l $ be the minimal positive integer satisfying $ q^{2l}=1 $. The module $ \V_{n,p} $ (or the pair $ (n,p) $) is called critical if $ q^{2(n-2p+1)}=1 $.

\begin{theorem} \cite[Section 7.3, Theorem 2]{M4}\textbf{.}  \label{thm1:cellhom2}
If $ 0 \leq p_1 - p_2 < l $ and $  n-p_1 - p_2 +1 = 0 \pmod{ l } $, then there is a non-trivial homomorphism $ \theta :  \V_{n,p_2} \rightarrow \V_{n,p_1} $. Furthermore, the kernels and co-kernels of the homomorphism $ \theta $ are irreducible. Otherwise, there is no non-trivial homomorphism from $ \V_{n,p_2} $ to $ \V_{n,p_1} $.
\end{theorem}

 Define $ r_{(n,p)} $ be the integer satisfying the equation $ n-2 p +1 = k l + r_{(n,p)} $, where $ k \in \mathbb{N} $ and  $ r_{(n,p)} \in \{ 1, \dots , l \} $. The critically of $ (n,p) $ is equivalent to $ r_{(n,p)}=l $.

\begin{proposition} \cite[Section 7.3, Theorem 2]{M4}\textbf{.} \label{thm1:radheadDim}
Let $ q $ be a root of unity and $ (n,p) $ be non-critical. Then 
\begin{align}
\dim \RV_{n,p, \delta }  = \left\lbrace \begin{array}{ll}
\dim \Lcell_{n, \,  p+r_{(n,p)}-l, \, \delta}     & \text{if} \; p+r_{(n,p)} -l \geq 0, \\
0   & \text{otherwise}.
\end{array} \right. 
\end{align} 
\end{proposition}

\section{The bubble algebra $ \Tm $}

 \hspace{15pt}
Throughout the paper, let $ n,m $ be positive integers, $ \Cc_0 ,   \dots , \Cc_{m-1} $ be different colours where none of them is white, and  $ \delta_0,  \dots  ,  \delta_{m-1}  $ be scalars corresponding to these colours.

The aim of this section is introducing the multi-colour partition algebra and then defining the bubble algebra as a sub-algebra.

Define the set $ \Phi^{n,m} $ to be 
\[ \{ (A_0, \dots , A_{m-1}) \mid \{ A_0, \dots, A_{m-1} \} \in \mathcal{P}_n \}.\]

We construct basis elements of  the multi-colour partition algebra  in similar way of the algebra $ \Pp $. Let $ (A_0, \dots , A_{m-1}) \in \Phi^{n,m} $ (note that some of these subsets can be an empty set). Define $  \mathcal{P}_{A_0 , \ldots ,  A_{m-1}} $ to be the set $ \prod_{i=0}^{m-1} \mathcal{P}_{A_i}  $, where $ \mathcal{P}_{A_i}  $ is the set of all partitions of $ A_i$, and 
\begin{align*}
\ppm :=&  \bigcup\limits_{ (A_0, \dots , A_{m-1}) \in \Phi^{n,m} }  \mathcal{P}_{A_0 ,  \ldots ,  A_{m-1}} .
\end{align*}

The element $ d= (d_0, \dots , d_{m-1} ) \in \prod_{i=0}^{m-1} \mathcal{P}_{A_i} $ can be represented by the same diagram of the partition $ \cup_{i=0}^{m-1} d_i \in \mathcal{P}_{n} $ after colouring it as follows: we use the colour $ \Cc_{i} $ to draw all the edges and the nodes in the partition $ d_i $.

A diagram represents an element in $ \ppm $ is not unique.  We say two diagrams are equivalent if they represent the same tuple of partitions. The term multi-colour partition diagram will be used to mean an equivalence class of a given diagram. For example, the following diagrams
\begin{center}
\begin{tikzpicture}
          \draw[gray] (0,0) rectangle (1.25,0.75);
           \fill[red] (0.25,0)  circle[radius=1.5pt];
            \fill[red] (0.25,0.75)  circle[radius=1.5pt];
            \fill[blue] (0.5,0)  circle[radius=1.5pt];
            \fill[blue] (0.5, 0.75)  circle[radius=1.5pt];
            \fill[blue] (0.75,0)  circle[radius=1.5pt];
            \fill[red] (0.75, 0.75)  circle[radius=1.5pt];
            \fill[red] (1, 0.75)  circle[radius=1.5pt];
            \fill[blue] (1,0)  circle[radius=1.5pt];
            \draw[red] (0.25,0)--(0.25, 0.75 );
            \draw[blue] (0.5, 0.75) .. controls ( 0.6,0.3) and (0.95,0.3) .. (1, 0);
            \draw [domain=180:360][red] plot ({0.25*cos(\x)+0.5}, {0.2*sin(\x)+0.75});
            \draw[ domain=0:180][blue] plot ({0.125*cos(\x)+0.625}, {0.15*sin(\x)+0});
            \draw[ domain=0:180][blue] plot ({0.125*cos(\x)+0.875}, {0.15*sin(\x)+0});
\end{tikzpicture} 
\;\;\;\;\;\;\;
\begin{tikzpicture}
          \draw[gray] (0,0) rectangle (1.25,0.75);
           \fill[red] (0.25,0)  circle[radius=1.5pt];
            \fill[red] (0.25,0.75)  circle[radius=1.5pt];
            \fill[blue] (0.5,0)  circle[radius=1.5pt];
            \fill[blue] (0.5, 0.75)  circle[radius=1.5pt];
            \fill[blue] (0.75,0)  circle[radius=1.5pt];
            \fill[red] (0.75, 0.75)  circle[radius=1.5pt];
            \fill[red] (1, 0.75)  circle[radius=1.5pt];
            \fill[blue] (1,0)  circle[radius=1.5pt];
            \draw[red] (0.25,0)--(0.27, 0.75 );
            \draw[blue] (0.5, 0.75)--(0.5, 0);
            \draw [domain=180:360][red] plot ({0.25*cos(\x)+0.5}, {0.2*sin(\x)+0.75});
            \draw[ domain=0:180][blue] plot ({0.25*cos(\x)+0.75}, {0.2*sin(\x)+0});
            \draw[ domain=0:180][blue] plot ({0.125*cos(\x)+0.875}, {0.1*sin(\x)+0});
\end{tikzpicture} 
\end{center}
are equivalent.

 We define the following sets for each element $ d \in \prod \mathcal{P}_{A_i} $:
\begin{align*} 
 \Top(d_i)= A_i \cap\underline{n} , \;\;\;\;\;\;\;\;\;\;\;\;\;\;\;\;\;\;\;\;\;\;\;\;  \Bot(d_i) =  A_i \cap\underline{n}' , \;\;\;\;\;\;\;\;\;\;\;\;\;\;\;\;\;\;\;\;\;\;\;\; \\
 \Top(d)= ( \Top(d_0) ,\dots ,\Top(d_{m-1}) ), \;\;\;\;\;\;\;\;  \Bot(d) =( \Bot(d_0) ,\dots ,\Bot(d_{m-1}) )  .
\end{align*} 

\begin{definition} 
Let $ \Pm $ be $\F$-vector space with the basis $ \ppm $ and with the composition:
\begin{align*} 
( \alpha_i)( \beta_i) = \left\{ 
   \begin{array}{l l}
    \prod\limits_{i=0}^{m-1} \delta_i ^{c_i}  ( \beta_j \circ \alpha_j ) & \quad \text{if } \;  \Bot( \alpha)= \Top( \beta) ,\\
     0 & \quad \text{otherwise.}
   \end{array} \right. 
\end{align*}
where $\delta_i \in \F$, $ \alpha, \beta \in \ppm $, $c_{i}$ is the number of removed connected components from the middle row when computing the product $ \beta_i  \circ \alpha_i $ for each $ i=0, \dots , m-1 $ and $ \circ $ is the normal composition of partition diagrams.
\end{definition}

\begin{proposition}
The product on $ \Pm $ that defined in the previous definition, is associative.
\end{proposition}
\begin{proof}
Let $ \alpha= ( \alpha_i) ,$  $ \beta= ( \beta_i) $ and $ \rho = ( \rho_i ) $ be multi-colour partitions in $ \ppm $. 
Note that $ \Top( \alpha_i \circ \beta_i )= \Top( \beta_i )$ and $ \Bot( \alpha_i  \circ \beta_i )= \Bot( \alpha_i )$ as long as $ \alpha \circ \beta $ is defined. From the multiplication on $ \PM $, the composition $ \alpha \circ ( \beta  \circ \rho  ) $ is defined if and only if $ \Top( \alpha_i )= \Bot( \beta_i \circ \rho_i )$ for each $i$, and $ \beta  \circ \rho $ is defined if and only if $ \Top( \beta_i )= \Bot( \rho_i )$ for each $i$. But if $ \beta  \circ \rho $ is defined then $\Bot( \beta_i \circ \rho_i )= \Bot( \beta_i)$ for each $i$. Then $ \alpha \circ ( \beta  \circ \rho  ) $ is defined if and only if $ \Top( \alpha_i )= \Bot( \beta_i)$ and $ \Top( \beta_i )= \Bot( \rho_i )$ for each $ i $. Similarly, $( \alpha \circ  \beta)  \circ \rho  $ is defined if and only if $ \Top( \alpha_i )= \Bot( \beta_i)$ and $ \Top( \beta_i )= \Bot( \rho_i )$ for each $i$. So the composition $ \alpha \circ ( \beta  \circ \rho  ) $ is defined if and only if $( \alpha  \circ \beta ) \circ \rho $ is defined, then the product in $ \Pm$ is an associative when vanishes. 
Furthermore, if it does not vanish, we have
\begin{align*}
 (  \rho \beta  ) \alpha =  & \big(  \alpha_0 \circ ( \beta_0 \circ \rho_0 ), \dots , \alpha_{m-1} \circ ( \beta_{m-1} \circ \rho_{m-1} ) \big) \\
= & \big( ( \alpha_0 \circ \beta_0 ) \circ \rho_0 , \dots,  ( \alpha_{m-1} \circ  \beta_{m-1} ) \circ \rho_{m-1} \big) = \rho ( \beta \alpha ) ,
\end{align*}
as the composition of partition diagrams is associative.
\end{proof}

From the previous proposition, we have $ \Pm$ is an associative algebra with identity:
\[ 1_{ \PM } = \sum\limits_{(A_0, \dots , A_{m-1}) \in \Xi^{n,m} } ( 1_{A_0} ,  \ldots ,  1_{A_{m-1}} ) , \]
where $ \Xi^{n,m} := \{ (A_0 , \ldots ,  A_{m-1}) \mid \cup_{i=0}^{m-1} A_i= \underline{n}, A_i \cap A_j = \emptyset \; \forall i \neq j \} $, $1_{ A_i }$ is the partition of the set $ A_i \cup A'_i $ where any node $ j $ is only connected with the node $ j' $  for all $ j \in A_i$ and $ A_i'=\{ j' \mid j \in A_i \} $, for all $ 0 \leq i \leq m-1 $. This means the identity is the summation of all the different multi-colour partitions that their diagrams connect $i$ only to $i'$ with any colour for each $1 \leq i \leq n $. The algebra $ \Pm $ is called the multi-colour partition algebra.

\begin{definition} \cite[Section 2]{GM2}\textbf{.} 
The propagating number of $ \alpha \in \ppm $, $ \# (\alpha) $, is the number of parts which contain nodes from both the top and the bottom rows in any colour, i.e. $\#( \alpha )=\sum_{i=0}^{m-1} \#( \alpha_i)  $ or simply $\# (\alpha)= \# \big( \cup_{i=0}^{m-1} \alpha_i  \big) $.
\end{definition}
\begin{definition} \cite[Section 2]{GM2}\textbf{.}
The $ \Cc_j $- propagating number of $ \alpha \in \ppm $, $ \#_j(\alpha) $, is the propagating number of $ \alpha_{j} $.
\end{definition}

The propagating number of diagrams in the algebra $ \PM $ has similar property of propagating number of diagrams in $ \Pp $: if $ \alpha, \beta\in \ppm $ with $ \alpha \beta \neq 0$, then
\begin{align*}
  \# (\alpha \beta)  \leq \min( \#(\alpha), \#(\beta)) ,  \;\;\;\;\;\;\;  \#_j( \alpha \beta)  \leq \min( \#_j(\alpha), \#_j(\beta)).  
\end{align*}

A planar multi-colour partition in the set $ \ppm $ is a multi-colour partition whose a diagram that does not have edge crossings in the same colour. This is the same definition that Grimm and Martin use in \cite{GM2}. In other words, there can be crossed edges but they don't have the same colour.  This definition of planar diagram is consistent with the definition of planar diagram in the algebra $ \Pp$ provided that considering all the diagrams in $ \Pp$ have been coloured by using only one colour.

We define subsets of $\ppm $ corresponding to those subsets of $\mathcal{P}_n$, as following:
\begin{align} \label{eq2:ColourSets}
  \left.\begin{array}{r@{\mskip\thickmuskip}l}
 \Sm &= \{ d\in \ppm \mid \# (d)=n \}, \\
\mathcal{A}_{n,m} &= \{ d\in \ppm \mid  d \text{ is planar} \},\\
\mathcal{B}_{n,m} & =\{ d\in \ppm \mid \, \text{all blocks of }  d  \text{ have size } 2 \},  \\
 \mathcal{T}_{n,m} & =\mathcal{A}_{n,m} \cap \mathcal{B}_{n,m}  ,\\
 \widehat{\Sym}_{n,m}&= \Sm \cap \mathcal{A}_{n,m}.
  \end{array} \;\;\; \right\}
\end{align}

The diagrams in the bubble algebra, as Grimm and Martin\cite{GM2} defined them, in the case of two colours can be constructed by drawing two Kauffman diagrams (or just one) with no internal loops, using different colours in the same frame with $n$ nodes on the northern face and $n$ nodes on the southern face, such that if a node is contained in first Kauffman diagram, it will not be contained in the second. This means that at these diagrams the nodes are connected in pairs with different colours where an intersection is just allowed between different colour edges.

The bubble algebra $\Tm $ (it is denoted by $ T_{n}^2 (\delta _r,\delta _b) $ in \cite{GM2} in the case of two colours), or simply $\TM $  and $ \TM ( \breve{ \delta } )$ for simplicity where $ \breve{\delta}= ( \delta_0 , \dots , \delta_{m-1}) $, is the $\F$-linear extension of the set of isotopy classes of previous diagrams and composition defined as the one on $ \Pm $, with internal closed loop replacement. The loop replacement scalar  here depends on the colour.

\begin{theorem}
The bubble algebra $ \Tm $ is the sub-algebra of the algebra $ \Pm$ spanned by the set $\mathcal{T}_{n,m} $, which is defined in equation \eqref{eq2:ColourSets}.
\end{theorem}
\begin{proof}
 From the description of diagrams in the bubble algebra $ \Tm $, we can identify bubble diagrams with multi-colour partitions and hence define it as sub-algebra of the algebra $ \Pm$. We are going to show that $ \tm $ is closed under the composition on $ \PM $ and then the rest follows immediately from the algebra $ \PM $ and from bubble diagrams realisation and from the fact $ 1_{(A_0, \dots , A_{m-1})} \in \tm $ for each $ (A_0, \dots , A_{m-1}) \in \Xi^{n,m}  $. Let $ D=(D_0, \dots, D_{m-1}) $ and $ B=(B_0, \dots , B_{m-1}) $ be multi-colour partitions in $ \tm $ such that $ BD \neq 0 $, so we have $ D_i \circ B_i $ is defined as partition diagrams. Now from the definition of $ \tm $, all the partitions $ D_i $ and $ B_i $ are representing by Kauffman diagrams, but then $ D_i \circ B_i $ is also representing by Kauffman's diagram for each $ i $, and we are done.
\end{proof}

\section{Cell modules}

 \hspace{15pt}
Making an arc, an edge connects two nodes in the same row (top or bottom) of a diagram,  needs two vertices on this row, so the propagating number of any diagram $ d \in \tm $ has the form $ \# (d) = n-2v $ for some integer $ v $, where $ 0 \leq v \leq [ ^n/_2 ] $.

 Define the set $\Gamma_{(l,m)} $ to be 
 \[ \{ \lambda=(\lambda_0, \dots , \lambda_{m-1} ) \mid  \lambda_i \in \mathbb{N} \cup \{ 0 \} \; \text{for each} \; i \; \text{and} \;   \sum_{i=0}^{m-1} \lambda_i = l  \} \]
  and the set 
 \[ \Lambda := \bigcup_{ v=0}^{ [^n/_2] } \Gamma _{(n-2v,m)} . \]

We follow Grimm and Martin \cite{GM2} and define the subset $ \tm [ \lambda_0, \ldots , \lambda_{m-1}]$, or simply $ \tm [ \lambda ]$, to be $ \{ d \in \tm | \# _{j}(d)= \lambda_{j} \; \text{ for all} \;  j \in \mathbb{Z}_m \} $, where $ \lambda \in \Lambda $.

A half-multi-colour diagram, or simply half-diagram, is a diagram obtained by cutting horizontally a diagram in the set $ \tm $ in the middle such that each propagating line is cut once, thus this is well defined on classes. As for the Temperley-Lieb algebra, we can form a  unique bubble algebra diagram from two half-diagrams providing that they have the same number of propagating lines of each colour.  Let $ \tm^{| \rangle } [ \lambda ]$ be the set of top pieces obtained by cutting elements of the set $ \tm [ \lambda ]$, where $ \lambda \in \Lambda $. Similarly $ \tm^{\langle |} [ \lambda ]$ is the set of bottom pieces obtained by cutting elements of $\tm [ \lambda ]$.

 A half-diagram is called a $ ((n_0,p_0), \ldots , (n_{m-1} , p_{m-1})) $-link state, if it contains both $n_j$ nodes and  $p_j$ arcs of the colour $ \Cc_j $ for each $ j $. This means that there are $ n_j-2p_j $ unconnected nodes of the colour $\Cc_{j}$ for each $j$.

 Denote by $ \F \M_n ( \lambda_0, \dots , \lambda_{m-1} ) $, or simply $ \F \M_n ( \lambda ) $ where $ \lambda \in \Lambda $, the vector space with a basis $ \M_n ( \lambda) $ which contains all link states that have number of defects of the colour $ \Cc_j $ on the form $ \lambda_j -2t_j $ for each $ j \in \Z_m $ where $ 0 \leq t_j \leq [^{\lambda_j} /_2 ] $. Note that there is no condition on the colours of arcs.

\begin{lemma} \label{lem5:1}
Let $ \lambda \in \Lambda$. The vector space $ \F \M_n (\lambda ) $ is a left $ \TM $- module with the action defined  by  the concatenation of diagram with a half-diagram then proceeding as we would with two diagrams in $ \TM$ (remove each loop and replace it by parameter corresponding to the loop's colour and it will be zero if they have different distribution of colours), and finally omit any new bottom arcs.  
\end{lemma}
\begin{proof}
Let $ x \in \tm $ and $ d $ be a half-diagram in $ \M_n ( \lambda ) $. Without loss of generality, we can assume $ xd \neq 0 $, multiplying $ x $ with $ d $ can not create any additional propagating lines of any colour. Thus the number of $ \Cc_j $-defects in $  x d $ is of the form $ \lambda_j -2t_j $ where $ 0 \leq t_j \leq [^{ \lambda_j} /_2 ] $, because making an extra $ \Cc_j $-arc needs two $ \Cc_j $-nodes.
\end{proof}

  Define a subset  $ \M_n^{ ^<} ( \lambda ) $ to be $ \bigcup_{j=0}^{m-1}   \M_n ( \lambda_0, \dots , \lambda_j-2, \dots , \lambda_{m-1} ) $. Note that $ \M_n ( \lambda_0, \dots , $  $\lambda_j-2, \dots , \lambda_{m-1} ) $ is taken to be the empty-set when $ \lambda_j < 2 $. Let  $ \F \M_n^{^<} ( \lambda ) $ be the module that generated by $  \M_n^{^<}  ( \lambda ) $, thus $ \F \M_n^{^<} ( \lambda ) $ is a sub-module of $ \F \M_n ( \lambda ) $.

\begin{lemma} \label{lem5:2}
Let $ \Delta_n ( \lambda ) $ be the module  $ \F \M_n ( \lambda  ) / \F \M_n^{<} ( \lambda  )$ of $ \TM $, where $ \lambda \in \Lambda $. Then the module $ \Delta_n ( \lambda ) $ has the set $ \tm^{| \rangle } [ \lambda  ] $ as a basis.
\end{lemma}
\begin{proof}
The quotient  $ \F \M_n ( \lambda ) / \F \M_n^{<} ( \lambda) $ sends any link state with less than $ \lambda_j $ defects of the colour $ \Cc_j $ for each $j$ to be zero. Thus the left multiplication by any diagram in the set $ \tm $ will be either zero or a half-diagram in $ \tm^{| \rangle } [ \lambda  ] $.
\end{proof}

\section{ Cellularity of the bubble algebra }

\begin{theorem}\cite[Proposition 1.3.2]{JM}\textbf{.}
 The algebra $ \TM ( \breve{\delta} ) $ is cellular over any field, with  the involution sending each diagram to its reflection in the horizontal plane, and the indexing set $ \Lambda = \bigcup_{ v=0}^{ [^n/_2] } \Gamma _{(n-2v,m)} $. The order on the set $  \Lambda $ is defined by 
\begin{center} 
 $ \lambda \geq  \lambda' $ f and only if $\lambda_j \leq  \lambda'_j $ for each $ j$.
\end{center}
The modules $ \Delta_n ( \lambda ) $ where $ \lambda \in  \Lambda $ are cell modules of the algebra $\TM$. 
\end{theorem}

Each cell module  $ \Delta_n ( \lambda ) $ comes with a contravariant inner product via its basis of top half-diagrams, defined as follows: let $ d,d' \in \tm [ \lambda ] $, $ x= \langle d| $ and $ y= | d' \rangle $, so
\[ dd'=|d \rangle \langle d | \; |d' \rangle \langle d' | =  \langle d| | d' \rangle \; |d \rangle \langle d' |  =:  \langle d| | d' \rangle d'', \]
so $ \langle x, y \rangle = \left\{ \begin{array}{ll}
\langle d| | d' \rangle \;\; & \text{ if} \; d'' \in \tm [ \lambda  ], \\
0 & \text{otherwise} .
\end{array} \right. $

Let $ \Gr_n ( \lambda ) $ to be the Gram matrix of the previous inner product on the cell module $ \Delta_n ( \lambda ) $ with respect to half-diagrams basis. Since we work over a field, we can check when the module $ \Delta_n ( \lambda ) $ is simple by computing $ \det \Gr_n ( \lambda ) $ as long as $ \langle \; , \; \rangle \neq 0 $. Grimm and Martin \cite{GM2} showed that the cell modules $ \Delta_n (\lambda_0, \lambda_1 ) $ are generically simple.

Let $ \Lambda^0 $ be subset of $ \Lambda  $ that contains all $ \lambda \in \Lambda  $ such that $ \langle \; , \; \rangle \neq 0 $. Note that when $ \delta_j \neq 0 $ for some $  j  $, then $ \Lambda ^0= \Lambda$, since we can take a half diagram with all the arcs of the colours corresponding to non-zero scalars. Even if $ \delta_j=0 $ for all $ j$, then for each cell module $ \Delta_n ( \lambda ) $ with $ \sum_{j=0}^{m-1} \lambda_j \neq 0 $, the inner product $  \langle \;  , \;  \rangle \neq 0 $ because we can still find diagrams such their product is equal to one.  Thus $ \Lambda^0 = \Lambda $ unless $ n $ is an even integer and $ \delta_i=0 $ for each $ i \in \Z_m $. In the case $ n $ is an even integer and $ \delta_i=0 $ for each $ i \in \Z_m $, then $ \Lambda^0 = \Lambda\setminus \{ (0, \dots , 0 ) \}$. Then the bubble algebra $ \TM ( \breve{\delta} ) $ is a quasi-hereditary if and only if $ \delta_j \neq 0 $ for some $ 0 \leq j < m $ or $ n $ is an odd integer.

 Let $ \TN^{+}(\breve{\delta}) $ be the subspace of $ \TN (\breve{\delta}) $ that spanned by all the diagrams in $ \tn $ which have an even number of blue-nodes on the top face. Since making an arc needs two nodes on the same face, thus the number of blue-nodes on the bottom face of the diagrams in $ \TN^{+}(\breve{\delta}) $ will be also an even number. The composition of two diagrams in $ \TN (\breve{\delta}) $ does not change the number of blue-nodes on top face of the first diagram, thus $ \TN^{+}(\breve{\delta}) $ is an algebra with an identity equals to the sum of all coloured images of $ \id \in \Sym_n $ that have an even number of blue-propagating lines. Similarly, define $ \TN^{-}(\breve{\delta}) $ to be the subspace of $ \TN (\breve{\delta}) $ that spanned by all the diagrams in $ \tn $ which have an odd number of blue-nodes on the top face. Also, $ \TN^{-}(\breve{\delta}) $ is an algebra with identity equals to the sum of all coloured image of $ \id \in \Sym_n $ that have an odd number of blue-propagating lines.

\begin{lemma} 
For any $ n >0 $, we have $ \TN (\breve{\delta})  =  \TN^{+}(\breve{\delta}) \oplus \TN^{-}(\breve{\delta}) $, as algebra.
\end{lemma}
\begin{proof}
This come from the fact any diagram in $ \tn $ will have even number or odd number of blue-nodes on the top face, so $ \TN (\breve{\delta})  = \TN^{+}(\breve{\delta}) +  \TN^{-}(\breve{\delta}) $. Furthermore, it is clear that $ \TN^{+}(\breve{\delta}) \cap \TN^{-}(\breve{\delta}) $ is zero and the product of any two elements from $ \TN^{+}(\breve{\delta})$ and $ \TN^{-}(\breve{\delta}) $ will be zero.
\end{proof}

As consequence of last lemma, to study the representations of the algebra $  \TN (\breve{\delta}) $, it is enough to study the representations of the algebras $ \TN^{+}(\breve{\delta}) $ and $  \TN^{-}(\breve{\delta}) $.

\section{Idempotent Localisations}

 \hspace{15pt} Let $ \mu \in \Gamma_{ (n,m) } $, we define $ \muu $ to be 
\[ (\{1, \dots, \mu_0 \} , \{1+\mu_0, \dots, \mu_0+\mu_1 \}, \dots, \{1+\sum_{j=0}^{m-2} \mu_j, \dots, n \} )  .\]

\begin{proposition}
Let $ (A_i) \in \Xi^{n,m} $ and $ \#_j ( 1_{(A_i)}) = \mu_j $ for each $ j $, then the elements $ 1_{(A_i)} $ and $ 1_{ \muu } $ are conjugate in the algebras $ \TM $ and $ \PM$.
\end{proposition}
\begin{proof}
 To show that we need to define an invertible element $ D \in \TM $ such that $ D^{-1} 1_{ (A_i)} D =  1_{ \muu } $. Claim that the element 
\[  \theta^{ (A_i) }+ \sum_{ B \in \Xi^{n,m} \setminus \{ (A_i)\} } 1_{ B } \] 
satisfies the previous equation, where $ \theta^{ (A_i) } $ is the multi-colour partition obtained from colouring a permutation $ \theta $ with top equals $ (A_i) $, and $ \theta $ is specific permutation changes the order of coloured lines without crossing lines that have the same colour.

Lets define the map $ \theta \in \Sym_n $ as follows:  assume that $ i \in \underline{n} $ and $ i \in A_j $, and define $ \theta(i) $ to be $ \nu_{i,j} + \sum\limits_{k <  j } \mu_k $, where $ \nu_{i,j} $ be the number of integers $ l \in \underline{n} $ that smaller than $ i $ and $ l \in A_j $.

We are going to show that $ \theta \in \Sym_n $, by proving that $ \theta $ is an injective map. It is obvious that $ \theta $ is well-defined.  Assume that $ i_1, i_2 \in \underline{n} $ without loss of generality we can say that $ i_1< i_2 $. Now there are two probabilities, $ i_1, i_2 \in A_j $ or $ i_1 \in A_{j_1} \neq A_{j_2} \ni i_2$. If $ i_1, i_2 \in A_j $ for some $ j $, then $ \nu_{i_1,j} <  \nu_{i_2,j} $ so $ \theta(i_1) < \theta(i_2) $. On the other side when $ j_1 \neq j_2 $, if $ j_1 < j_2 $, so $ \theta(i_1)= \nu_{i_1,j_1} + \sum\limits_{k < j_1 } \mu_k \leq \sum\limits_{k < j_1+1 } \mu_k < \theta(i_2)$. Similarly, if $ j_2 < j_1 $, thus  $ \theta(i_2) < \theta(i_1)$. Therefore $ \theta $ is an injective.

From the way that we define $ \theta $, it is evidential that $ \theta^{(A_i)} \in \TM $  since if $ i , j \in A_h $ for some $ h $ where $ i<j $, so $ \theta(i) < \theta(j) $ this implies that there is no crossing lines with same colour. Similarly, the diagram $ \big( \theta^{-1} \big) _{ (A_i) }$, the coloured image of  $ \theta^{-1} $ with bottom equals $ (A_i) $, is contained in $ \TM $ because by flipping the diagram $ \big( \theta^{-1} \big) _{ (A_i)}$ we obtain $ \theta^{(A_i)}$. Also note that $ \Bot ( \theta^{ (A_i) } ) = \muu $.

Finally, take $ D= \theta^{ (A_i)}+ \sum\limits_{B \in \Xi^{n,m} \setminus \{ (A_i)\} } 1_{B} $ and $ D'= \big( \theta^{-1} \big) _{ (A_i) }+ \sum\limits_{ B \in \Xi^{n,m} \setminus \{ \muu \}} 1_{B}$.  Note that $ D, D' \in \TM $, $ DD'=1_{\TM}=D'D $ and $ D 1_{(A_i) } D' = 1_{ \muu }$. 
\end{proof}

Jegan\cite{JM} proved in Theorem 3.1.4, for any $ \mu \in \Gamma_{ (n,m) } $ the algebras $ \bigotimes_{i=0}^{m-1}  \mathsf{TL}_{ \mu_i} ( \delta_i) $ and $ 1_{ \muu } \TM( \breve{ \delta} ) 1_{ \muu } $ are isomorphic with a map sending any tuple of diagrams in  the algebra $ \bigotimes_{i=0}^{m-1}  \mathsf{TL}_{ \mu_i} ( \delta_i) $ to the diagram in $ 1_{ \muu } \TM( \breve{ \delta} ) 1_{ \muu } $ formed by drawing these diagrams in one frame one by one using different  colours such that the diagram from $\mathsf{TL}_{ \mu_i} ( \delta_i) $ is drawn in the colour $ \Cc_i $. Similarly, if $ \V_{\mu_0,p_0}, \dots , \V_{\mu_{m-1},p_{m-1} }$ are cell modules for the algebras $\mathsf{TL}_{ \mu_0} ( \delta_0) , \dots ,  \mathsf{TL}_{ \mu_{m-1} } ( \delta_{m-1} ) $ respectively, then elements of the module $ \bigotimes_{i=0}^{m-1}  \V_{\mu_i,p_i} $ can be represented by $ \big( (\mu_i,p_i) \big)_{i \in \Z_m} $- link states, see Figure \ref{fig5:3}, by using the same map which it is the same isomorphism that Jegan used in the proof of the fact: let $ \lambda \in \Lambda $ and $ \mu \in \Gamma_{ (n,m) } $, then 
\begin{align*}
1_{ \muu }  \Delta_n ( \lambda ) & \cong \left\lbrace
 \begin{array}{ l l }
  \bigotimes\limits_{j=0}^{m-1}   \V_{ \mu_j ,t_j}   & \quad \text{ if} \; \mu_j - \lambda_j = 2t_j \; \text{ for each}  \; j  \; \text{for some }\;  t_j \in \mathbb{N}, \\ 
 0  & \quad \text{otherwise},
\end{array} \right.
\end{align*}
as $ 1_{ \muu } \TM 1_{ \muu } $-module.

\begin{figure}[!th] 
\caption{Illustration of a map from $ \V_{3,1} \otimes \V_{2,1} $ to $ 1_{(0,0,0,1,1)} \Delta_5(1,0) $ }
\label{fig5:3}
  \centering
    \includegraphics[width=60mm,scale=1.2]{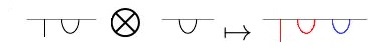}
\end{figure}

  \underline{ \textit{Important convention} }:  whenever $ \bigotimes\limits_{i=0}^{m-1}  \V_{\mu_i,p_i} $ or $ \bigotimes\limits_{i=0}^{m-1}  M_i $ are mentioned, where $ M_i $ is a sub-module or quotient module of $ \V_{\mu_i,p_i} $, we mean their image in $ 1_{ \muu } \Delta_n ( \lambda ) $ under the previous isomorphism.

A basis of $ \Delta_n ( \lambda ) $ is the set that contains all $ \big(( \lambda_j+2p_j,p_j)\big)_{j \in \Z_m} $-link states where $ p_0 , \dots , p_{m-1}$ are non-negative integers such that $ \sum_{j \in \mathbb{Z}_m } ( \lambda_j+2p_j)=n $, which is the same as the basis $ \tm^{| \rangle } [ \lambda  ] $. Each $ \big((n_j,p_j)\big)_{j \in \Z_m} $-link state  determines a collection of $(n_j,p_j)$-link states as they are defined in Section \ref{sec1}, where each $j$ represents the colour $ \Cc_j $, by omitting all the parts that have colour not $ \Cc_j $, thus
\[  \Delta_n ( \lambda ) = \sum_{ u \in \Gamma_{(v ,m)} } \; \sum_{ \sigma  \in \widehat{\Sym}_{n,m} } \sigma \big( \,  \bigotimes_{i=0}^{m-1}  \V_{ \lambda_i+2u_i,u_i } \, \big), \]
where $ \lambda \in \Gamma_{(n-2v,m)} $. For example, take $ \alpha $ to be the $ ((3,1),(2,0),(4,1)) $-link state  \includegraphics[width=30mm]{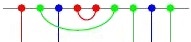}, so $ \alpha $ can be consider as a collection of the next link states:   
\begin{figure}[h] 
  \centering
    \includegraphics[width=60mm]{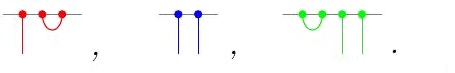}
\end{figure}

Let $ a = | D \rangle \in \Delta_n ( \lambda )  $ for some $ D \in \mathcal{T}_{n,m} [ \lambda ]$. The distribution of the colours of $ a $ is the set $ \Top ( D) $. This set will be denoted by $ \Top ( a ) $.

 Let $a$ be a $ \big( ( \lambda_j+2p_j,p_j) \big)_{j \in \Z_m} $-link state and $b$ be a $ \big( ( \lambda_j+2p'_j,p'_j) \big)_{j \in \Z_m} $-link state where $  \sum_{j \in \mathbb{Z}_m } p_j= \sum_{j \in \mathbb{Z}_m } p'_j $. It is evident that  $ \langle a, b \rangle =0 $ unless $ p_j=p'_j $ for each $ j$ and the distributions of the colours of $a$ and $b$ are same. When $ p_j=p'_j $ for each $ j$ and $ \Top ( a)=\Top ( b) $, and $a_{ j}$ be the $(n_j,p_j)$-link state which is obtained from $ a $ by omitting all the parts that have colour not $ \Cc_j $. Similarly, we define $ b_{ j } $. From the graphical visualization of the product on the algebra $ \TM$, we obtain 
 \begin{align} \label{eq5:InnerProd}
 \langle a, b \rangle & = \langle a_{ 0 } , b_{ 0}  \rangle_{n_0,p_0, \delta_0} \times  \cdots \times  \langle a_{ m-1}, b_{ m-1 } \rangle_{n_{m-1},p_{m-1}, \delta_{m-1}}  \; ,
 \end{align}
where $ \langle a_j, b_j \rangle_{n_j,p_j, \delta_j }$ denotes the standard bilinear form on $ \V_{n_j, p_j} $ as $ \mathsf{TL}_{n_j}( \delta_{j}) $-module. Note that distribution of colours, if it matches up, does not play any rule. In other words, if $a, b, c$ and $d$ be $ \big( (n_j,p_j) \big)_{j \in \Z_m} $-link states such that $ a_{ j }=c_{ j }$ and $ b_{ j }= d_{ j }$, then $  \langle a, b \rangle =  \langle c,d \rangle  $ if $ \Top ( a ) =\Top ( b ) $ and $ \Top ( c ) = \Top ( d )  $. Note that $ a $ and $ c $ may have different distributions of colours. As consequence of this, we have the following theorem.
\begin{theorem} \cite[lemma 3.2.10]{JM}\textbf{.} \label{thm5:77}
If $ \sum_{j \in \mathbb{Z}_m  } \lambda_j =n-2v $ for some $v$, then the Gram matrix of the cell module $ \Delta_n ( \lambda ) $ of the previous inner product with respect to half-diagrams basis can be written in the form
\[ \Gr_n ( \lambda ) = \bigoplus\limits_{ u \in \Gamma_{(v,m)} } \bigoplus\limits^{ n_{ \lambda+2u }}    \Gr_{ \lambda_0+2u_0,u_0, \delta_0 } \otimes  \dots \otimes  \Gr_{ \lambda_{m-1} +2u_{m-1},u_{m-1}, \delta_{m-1} } , \]
where $ \Gr_{ \lambda_j+2u_j,u_j, \delta_j } $ is the Gram matrix of the cell $ \mathsf{TL}_{ \lambda_j+2u_j} ( \delta_j ) $-module $ \V_{\lambda_j+2u_j,u_j} $ with a specific bilinear form and half-diagrams basis. Therefore, the determinant of Gram matrix is
\begin{align*}
 \det \Gr_n (\lambda ) = \prod\limits_{ u \in \Gamma_{(v,m)} }  \Bigg( \prod_{j=0}^{m-1} ( \det   \Gr_{ \lambda_j+2u_j,u_j, \delta_j })^{ \mathsf{d}_{ \lambda_j+2u_j,u_j}^{-1}} \Bigg)^{ ( \prod_{j=0}^{m-1} \mathsf{d}_{ \lambda_j+2u_j,u_j} ) \cdot  n_{ \lambda+2u } } \; , \nonumber
\end{align*}
 where $ \mathsf{d}_{ \lambda_j+2u_j,u_j}= \dim \V_{\lambda_j+2u_j, u_j} $ and $ n_\mu := \binom{n}{ \mu_0, \ldots , \mu_{m-1} }  $ for each $  \mu \in \Gamma_{( n,m)} $. 
\end{theorem} \qed

If $ \sum_{ j }  \lambda_j =n $, from the last theorem we have $ \Gr_n( \lambda  )=   \bigoplus^{n_{ \lambda} } (1) = I_{n_{\lambda} \times n_{\lambda} } $, where $ I_{n_{\lambda} \times n_{ \lambda } } $ is the identity matrix, so the module $ \Delta_n  ( \lambda  )$ is simple whenever $ \sum_{ j }  \lambda_j =n $. Also, if $  \delta_j=q_j+q_j^{-1} \neq 0 $ for all  $ j \in \mathbb{Z}_m $ and $ q_j $ is not a root of unity for any $ j $, then the algebra $ \TM ( \breve{\delta} ) $ is semi-simple.

\begin{proposition}
Let $ \lambda \in \Lambda^0 $. The head of the module $ \Delta_n( \lambda ) $ where $ \lambda \in \Gamma_{(n-2v,m) } $ for some $ v $, denoted by $ \Lcell_n ( \lambda ) $, satisfy the relation
\begin{align} \label{eq5:002}
\dim \Lcell_n ( \lambda ) =&  \sum\limits_{ u \in \Gamma_{(v,m)} } n_{ \lambda+2u } \prod\limits_{i=0}^{m-1}  \dim \Lcell_{ \lambda_i+2u_i,u_i, \delta_i  } , 
\end{align} 
where $ \Lcell_{ \lambda_i+2u_i,u_i , \delta_i } $ is the head of the $ \mathsf{TL}_{\lambda_i+2u_i}( \delta_i) $-module $ \V_{ \lambda_i+2u_i,u_i } $.
\end{proposition}
\begin{proof}
This follows from the fact that $ \dim \Lcell_n ( \lambda ) = \rank ( \Gr_n( \lambda) )  $ since our algebra is over a field and by using Theorem \ref{thm5:77} and properties of the rank of matrices.
\end{proof}

\begin{corollary}
Let $ \lambda \in \Lambda^0 $. The module  $ \Lcell_n( \lambda  )  $ decomposes as
\begin{align*}
  \bigoplus_{ u \in \Gamma_{(v,m)} } \bigoplus^{n_{ \lambda+2u }}  \Lcell_{ \lambda_0+2u_0,u_0 , \delta_0 }  \otimes \cdots \otimes  \Lcell_{ \lambda_{m-1} +2u_{m-1},u_{m-1} , \delta_{m-1} }  , 
\end{align*}
as a vector space, where $ \lambda \in \Gamma_{(n-2v,m)} $ for some $v $.
\end{corollary}
\begin{proof}
It comes directly from the fact that any two vector spaces have the same dimension they are isomorphic.
\end{proof}

\begin{lemma}  \label{thm5:dimRad} 
The dimensions of $ \Rad ( \Delta_{n}( \lambda_0 , \lambda_1 )) $, the radical of $ \Delta_{n}( \lambda_0 , \lambda_1 ) $, is
\begin{align}
   \sum\limits_{ u \in \Gamma_{(v,2)} } n_{ \lambda+2u } \bigg(   \dim   \RV_{ \lambda_0+2u_0,u_0 , \delta_0 } \dim \V_{ \lambda_1 +2u_1,u_1 } + \dim \V_{ \lambda_0+2u_0,u_0 }  \dim  \RV_{ \lambda_{1} +2u_1,u_1 , \delta_1 }   \nonumber \\
\;\;\;\;  -  \dim  \RV_{ \lambda_0+2u_0,u_0 , \delta_0 }  \dim \RV_{ \lambda_1 +2u_1,u_1 , \delta_1  } \big), \nonumber
\end{align} 
where $ \lambda \in \Gamma_{(n-2v,2)} $ and $ \RV_{ \lambda_i+2u_i,u_i , \delta_i } $ is the radical of the $ \mathsf{TL}_{\lambda_i+2u_i} ( \delta_i )$-module $ \V_{ \lambda_i+2u_i,u_i } $.
\end{lemma}
\begin{proof}
Since $ \dim \Rad ( \Delta_n( \lambda )  ) = \dim \Delta_n( \lambda )  - \dim \Lcell_n ( \lambda ) $, so $ \dim \Rad ( \Delta_n( \lambda )  ) $ equals 
\[ \sum\limits_{ u \in \Gamma_{(v,2)} } n_{ \lambda+2u } \big( \dim \V_{ \lambda_0+2u_0,u_0 }  \dim \V_{ \lambda_1 +2u_1,u_1 }   -    \dim \Lcell_{ \lambda_0+2u_0,u_0 , \delta_0 } \dim \Lcell_{ \lambda_1+2u_1,u_1 , \delta_1  } \big) . \]
But $ \dim \Lcell_{ n,p , \delta  } = \dim \V_{n,p} - \dim  \RV_{n,p , \delta } $, so $ \dim \Rad ( \Delta_n( \lambda )  ) $ is 
\begin{multline*}
  \sum\limits_{ u \in \Gamma_{(v,2)} } n_{ \lambda+2u } \bigg(   \dim   \RV_{ \lambda_0+2u_0,u_0 , \delta_0 } \dim \V_{ \lambda_1 +2u_1,u_1 } + \dim \V_{ \lambda_0+2u_0,u_0 }  \dim  \RV_{ \lambda_1 +2u_1,u_1 , \delta_1 }   \\
 -  \dim  \RV_{ \lambda_0+2u_0,u_0, \delta_0 }  \dim \RV_{ \lambda_1 +2u_1 ,u_1 , \delta_1 } \big).  \qedhere  
\end{multline*} 
\end{proof}

\begin{theorem} \label{thm5:rad}
Let $ \lambda \in \Gamma_{(n-2v,2)} $ for some $ v $.  Then $ \Rad ( \Delta_n( \lambda_0, \lambda_1 )  )  $ decomposes as
\begin{align*}
  \bigoplus\limits_{ u \in \Gamma_{(v,2)} } \bigoplus^{n_{ \lambda+2u }} \big( \RV_{ \lambda_0+2u_0,u_0 , \delta_0 }  \otimes  \V_{ \lambda_1 +2u_1,u_1 } +  \V_{ \lambda_0+2u_0,u_0 }  \otimes  \RV_{ \lambda_1 +2u_1,u_1 , \delta_1 }  \big) , 
\end{align*}
as a vector space, and it is equal to 
\begin{align*} 
  \sum\limits_{ u \in \Gamma_{(v,2)} } \;  \sum\limits_{ \sigma  \in \widehat{\Sym}_{n,2} }    \sigma  \big( \,   \RV_{ \lambda_0+2u_0,u_0 , \delta_0 }  \otimes  \V_{ \lambda_1 +2u_1,u_1 }   +  \V_{ \lambda_0+2u_0,u_0 }  \otimes  \RV_{ \lambda_1 +2u_1,u_1 , \delta_1 }  \, \big) .
\end{align*}
 Remember by $ \RV_{ \lambda_0+2u_0,u_0 , \delta_0 }  \otimes  \V_{ \lambda_1 +2u_1,u_1 }  $ and $  \V_{ \lambda_0+2u_0,u_0 }  \otimes  \RV_{ \lambda_1 +2u_1,u_1 , \delta_1 }  $ we mean their images in the module $ 1_{  \underline{\lambda+2u} } \Delta_n ( \lambda) $.
\end{theorem}
\begin{proof}
First part comes directly from last lemma, since they have the same dimension, note that 
$$ ( \RV_{ \lambda_0+2u_0,u_0 , \delta_0 }  \otimes  \V_{ \lambda_1 +2u_1 ,u_1 }  ) \cap  (  \V_{ \lambda_0+2u_0,u_0 }  \otimes  \RV_{ \lambda_1 +2u_1 ,u_1, \delta_1  } ) =  \RV_{ \lambda_0+2u_0,u_0 , \delta_0 }  \otimes  \RV_{ \lambda_1 +2u_1,u_1, \delta_1 }  .$$ 

Now we are going to prove the second part. As we mentioned before we have
\[  \Delta_n ( \lambda ) = \sum_{ u \in \Gamma_{(v ,2)} } \sum_{ \sigma  \in \widehat{\Sym}_{n,2} } \sigma \big( \,  \V_{ \lambda_0+2u_0,u_0 }  \otimes  \V_{ \lambda_1 +2u_1,u_1 }  \, \big) . \]
Let $ y $ be a $ (( \lambda_0+2u'_0,u'_0),( \lambda_1+2u'_1, u'_1)) $-link state for some $ u' \in \Gamma_{(v ,2)} $, so from the last equation we can assume that $ y= \pi ( y_0 \otimes y_1 ) $ for some $ \pi \in \widehat{\Sym}_{n,2} $ and $ y_i $ is a $ ( \lambda_i+2u'_i,u'_i) $-link state for each $ i $, and let $ x $ be an element in $ \sigma  \big( \,  \RV_{ \lambda_0+2u_0,u_0 , \delta_0 } \otimes \V_{ \lambda_1 +2u_1,u_1 }  \,  \big) $ or in  $ \sigma  \big( \,  \V_{ \lambda_0+2u_0,u_0 }  \otimes  \RV_{ \lambda_1 +2u_1,u_1 , \delta_1 } \, \big) $ for some $ u \in \Gamma_{(v,2)} $ and some $  \sigma  \in \widehat{\Sym}_{n,2} $, so we can assume that $ x= \sigma ( x_0 \otimes x_1) $ where $ x_0 \in \RV_{ \lambda_0+2u_0,u_0 , \delta_0 } $ or $ x_1 \in \RV_{ \lambda_1+2u_1,u_1 , \delta_1 } $. If $ u \neq u' $ or $ \sigma \neq \pi $, this means the colour distributions of $ x $ and $ y $ are different, so from the definition of the multiplication on the algebra $ \TM $ we have $ \langle y, x \rangle =0 $. On the other hand, if $ u = u' $ and $ \sigma = \pi $, from equation \eqref{eq5:InnerProd} we have $ \langle y, x \rangle = \langle y_0, x_0 \rangle_{ \lambda_0+2u_0, u_0, \delta_0}  \langle y_1 , x_1 \rangle_{ \lambda_1+2u_1, u_1, \delta_1} $. But $ x_i \in \RV_{ \lambda_i+2u_i,u_i , \delta_i } $ for some $ i $, so $ \langle y_i, x_i \rangle_{ \lambda_i+2u_i, u_i, \delta_i} =0 $ for some $ i $. Hence $ \langle y, x \rangle =0 $ for each $ y \in \Delta_n ( \lambda ) $, which means $ x \in \Rad ( \Delta_n( \lambda )  ) $. Thus 
\[   \sum\limits_{ u  } \;  \sum\limits_{ \sigma  }    \sigma  ( \,  \RV_{ \lambda_0+2u_0,u_0 , \delta_0 }  \otimes  \V_{ \lambda_1 +2u_1,u_1 } +  \V_{ \lambda_0+2u_0,u_0 }  \otimes  \RV_{ \lambda_1 +2u_1,u_1 , \delta_1 }  \, ) \subseteq  \Rad ( \Delta_n( \lambda )  ) , \]
but both of them have the same dimension thus they are identical. 
\end{proof}

\begin{theorem}  \label{thm3:rad7}
Let $ \lambda \in \Gamma_{(n-2v,m)} $ for some $ v $.  Then the radical $  \Rad ( \Delta_n( \lambda )  ) $ equals
\begin{align*} 
\sum\limits_{ u \in \Gamma_{(v,m)} } \;  \sum\limits_{ \sigma  \in \widehat{\Sym}_{n,m} }   \sigma  \Big( \, ( \RV_{ \lambda_0+2u_0,u_0 , \delta_0 }  \otimes  \V_{ \lambda_1 +2u_1,u_1 } \otimes \cdots \otimes \V_{ \lambda_{m-1} +2u_{m-1},u_{m-1} })+   \\
( \V_{ \lambda_0+2u_0,u_0  }  \otimes  \RV_{ \lambda_1 +2u_1,u_1, \delta_1 } \otimes  \V_{ \lambda_2 +2u_2,u_2 } \otimes \cdots \otimes \V_{ \lambda_{m-1} +2u_{m-1} ,u_{m-1} })  +\;\;\;\;  \\
 \cdots + (\V_{ \lambda_0+2u_0,u_0 }  \otimes \cdots \otimes  \V_{ \lambda_{m-2} +2u_{m-2},u_{m-2} } \otimes \RV_{ \lambda_{m-1} +2u_{m-1} ,u_{m-1} , \delta_{m-1} })  \, \Big) .
\end{align*}
 Remember by the tensor product of the modules in the last equation we mean their images in the module $ 1_{ \underline{\lambda+2u} } \Delta_n ( \lambda) $.
\end{theorem}
\begin{proof}
We can show that by using induction on $ m $ and Theorem \ref{thm5:rad}.
\end{proof}

\begin{corollary} 
Let $ \lambda \in \Gamma_{(n-2v,m)} $, then 
\begin{align*}
\Lcell_n( \lambda ) =     \sum\limits_{ u \in \Gamma_{(v,m)} } \;  \sum\limits_{ \sigma  \in \widehat{\Sym}_{n,m} }    \sigma  \big( \,   \Lcell_{ \lambda_0+2u_0,u_0 , \delta_0 }  \otimes \cdots \otimes \Lcell_{ \lambda_{m-1} +2u_{m-1},u_{m-1} , \delta_{m-1} }  \, \big) .
\end{align*}
By $ \otimes_{i=0}^{m-1} \Lcell_{ \lambda_i+2u_i,u_i , \delta_i } $ we mean its images in the module $ 1_{ \underline{\lambda+2u} } \Delta_n ( \lambda) $.
\end{corollary}
\begin{proof}
From the definition of the head of cell module over a field and from the last theorem, we have $  \Lcell_n( \lambda )  $ equals
\begin{align*}
 \sum\limits_{ u \in \Gamma_{(v,m)} } \sum\limits_{ \sigma  \in \widehat{\Sm} }   \sigma \Big( \, \frac{  \V_{ 0 }  \otimes \cdots \otimes  \V_{  m-1 }   }{  \RV_{0 }  \otimes  \V_{ 1 } \otimes \cdots \otimes \V_{ m-1 } + \dots + \V_{ 0 }  \otimes \cdots \otimes  \V_{ m-2 } \otimes \RV_{ m-1 }  } \, \Big) ,
 \end{align*}
 where we put $ \RV_i:=  \RV_{ \lambda_i+2u_i,u_i , \delta_i }$ and $ \V_i:= \V_{ \lambda_i +2u_i ,u_i }  $ for simplicity.  Let $ x_i \in  \Lcell_{ \lambda_i+2u_i,u_i , \delta_i }  := L_i $ for each $ i $, so $ x_i = a_i + \RV_i $ for some $ a_i \in \V_i $ and from that we have
 \[ \otimes_{i=0}^{m-1} x_i = \otimes_{i=0}^{m-1} a_i + \RV_{0 }  \otimes  \V_{ 1 } \otimes \cdots \otimes \V_{ m-1 } + \dots + \V_{ 0 }  \otimes \cdots \otimes  \V_{ m-2 } \otimes \RV_{ m-1 }, \]
 it follows that 
\begin{align*}
\frac{  \V_{ 0 }  \otimes \cdots \otimes  \V_{  m-1 }   }{  \RV_{0 }  \otimes  \V_{ 1 } \otimes \cdots \otimes \V_{ m-1 } + \dots + \V_{ 0 }  \otimes \cdots \otimes  \V_{ m-2 } \otimes \RV_{ m-1 }  } =   \Lcell_{ 0 }  \otimes \dots \otimes  \Lcell_{ m-1 } ,
 \end{align*}
for each $ u \in \Gamma_{(v ,m)} $. We are done
\end{proof}

\section{Homomorphisms between cell $ \TM $-modules}

 \hspace{15pt}  As we said, the algebra $ \TM ( \breve{\delta} ) $ is semi-simple algebra when $ q_j $ is not root of unity where $ \delta_j =q_j+q_j^{-1} \neq 0 $ for each $ j \in \Z_m $. Therefore in what follows, it will be assumed that $ q_j $ is a root of unity for some $ j $, and let $ \mathbf{ l}_j $ be the minimal positive integer satisfying $ q_j^{2 \mathbf{ l}_j}= 1 $.

  The first part of next proposition is Lemma 4.1.1 in \cite{JM}. 

\begin{proposition} \label{thm5:ghom}
Let $ \lambda, \mu \in \Lambda_{ \TM } $ and $ \theta :  \Delta_n ( \lambda ) \rightarrow \Delta_n ( \mu ) $ be a homomorphism defined by $\theta ( a ) = \sum_{ i } \alpha_i  b_i $,  where $ \alpha_i \in \F $, $ a \in \tm^{ | \rangle }  [ \lambda ] $ and $ b_i  \in \tm^{ | \rangle }  [ \mu ]$ for each $ i $. Then the following are true:
\begin{enumerate}
\item $ \Top (a) = \Top (b_i )  $ for each $ i $.
\item $  \mu_j = \lambda_j -2t_j$, for some $ t_j \in \{ 0 , \dots , [^{ \lambda_j} /_2] \} $.
\item If $ \delta_j  $ is invertible and $ a $ contains an $ \Cc_j $-arc, then $ b_i $ contains an $ \Cc_j $-arc in the same position. This means that $ \theta $ preserves arcs when $ \delta_j  \neq  0 $ for each $ j \in \Z_m $.
\end{enumerate}
\end{proposition}
\begin{proof}
We are going to show only the last part. Assume that $ a $ contains $ h $ arcs of the colour $ \Cc_j $ and $ \delta_j \neq 0 $. Take $ x \in \TM $ to be the diagram defined as follows: $ \Top (x) = \Bot ( x) = \Top (a) $ and if any two nodes $ k,l \in \underline{n} $ are connected in $ a $ by a $ \Cc_j $-arc, then these nodes will be also connected in $ x $ by a $ \Cc_j $-arc and $ k', l' $ will be connected by the same colour, other that all the nodes will be connected to their projection in the bottom row. Note that $ x a = \delta_j^h  a $, so 
\[ \theta (a ) = \delta_j ^{-h} \sum_{i} \alpha_i x  b_i = \sum_{i} \alpha_i    b_i  . \]
The $ \Cc_j $-arcs on the top row will not be affected by the product, so they will be in $ x b_i  $ in the same positions of $ a $ for each $ i $.
\end{proof}

Let $ \lambda \in \Gamma_{(n-2v,m)} $ for some $ v $, and $ \theta : \Delta_n ( \lambda  ) \rightarrow  \Delta_n ( \lambda -2t ) $ be a homomorphism, where $ \lambda-2t = ( \lambda_0-2t_0,  \dots , \lambda_{m-1}-2t_{m-1} ) $. The homomorphism $ \theta $ will be non-zero if and only if there is $ \mu \in \Gamma_{(n,m)} $ of the form $ \mu= \lambda+2p $ for some $ p \in \Gamma_{(v,m)} $ such that $ \theta ( 1_{ \muu } \Delta_n ( \lambda  )) \neq \{ 0 \} $. Thus we can restrict $ \theta $ to define a non-trivial homomorphism 
\[  \theta_{\mu} : \otimes_{i=0}^{m-1} \V_{\mu_i,p_i} \longrightarrow  \otimes_{i=0}^{m-1} \V_{\mu_i,p_i+t_i} . \]
Note that if $ \delta_i \neq 0 $ for each $ i $, so $ p $ does not have any important role since it is corresponding to number of arcs which are actually  preserved, see Proposition \ref{thm5:ghom}. Furthermore, if we have a homomorphism from $ \otimes_{i=0}^{m-1} \V_{\lambda_i+2p_i,p_i} $ to  $ \otimes_{i=0}^{m-1} \V_{ \lambda_i+2p_i,p_i+t_i} $, we can extend it to get a homomorphism from $  \Delta_n ( \lambda ) $ to $ \Delta_n ( \lambda-2t) $. Thus 
\[ \Hom_{\TM} ( \Delta_n ( \lambda ) , \Delta_n ( \lambda-2t))= \{ 0 \} \]
 if and only if 
\[  \Hom_{\otimes_{i=0}^{m-1} \mathsf{TL}_{\mu_i}( \delta_i)} (  \otimes_{i=0}^{m-1} \V_{\mu_i,p_i} , \otimes_{i=0}^{m-1} \V_{\mu_i,p_i+t_i}) = \{ 0 \} \]
 for each $ p \in \Gamma_{(v,m)} $.

 Now, if there is a non-zero homomorphism $ f_i \in \Hom_{  \mathsf{TL}_{\mu_i}( \delta_i)} (   \V_{\mu_i,p_i} ,  \V_{\mu_i,p_i+t_i}) $ for each $ i $, then $ \otimes f_i \in \Hom_{\otimes_{i=0}^{m-1} \mathsf{TL}_{\mu_i}( \delta_i)} (  \otimes_{i=0}^{m-1} \V_{\mu_i,p_i} , \otimes_{i=0}^{m-1} \V_{\mu_i,p_i+t_i}) $ is also non-zero. From the previous details we have the following propositions.

\begin{proposition} \cite[ Theorem 6.2.2]{JM}\textbf{.} 
Let $ \delta_j $ is invertible for each $ j $, $ \lambda'=  \lambda-2t  $ where $ \sum_{j=0}^{m-1} \lambda_j =n-2v $ for some $ v $. Suppose there exist non-zero homomorphisms from $ \V_{\lambda_i,0} $ to $\V_{\lambda_i,t_i} $ as $ \mathsf{TL}_{\lambda_i} ( \delta_i ) $-modules for each $ i $. Then there exists a non-trivial homomorphism from $ \Delta_n ( \lambda  ) $ to $  \Delta_n ( \lambda' )$.
\end{proposition}

\section{The Cartan matrix of the bubble algebra}

 \hspace{15pt}
Throughout this section we assume that $ \delta_i= q_i+q_{i}^{-1} \in \mathbb{C} $ for each $ i $ and at least one of the parameters is a root of unity other than $ \pm 1 $. We aim to compute the decomposition matrix  of the algebra $ \TM $ over complex field, then the Cartan matrix for $ \TM $ can be found, since the bubble algebra is cellular.

\begin{proposition}
Let $ \lambda \in \Gamma_{(n-2v,m)} $ for some $ v   $. The module $ \Delta_n ( \lambda )  $ is simple if and only if $ \lambda_i+1=0 \pmod{ \mathbf{ l}_i }$ whenever $ q_i $ is a root of unity where $ i \in \Z_m $.
\end{proposition}
\begin{proof}
If $ q_i $ is not a root of unity for some $ i\in \Z_m $, Proposition \ref{thm1:TLsemisimple} implies to $ \Lcell_{ \lambda_i+2u_i, u_i, \delta_i } = \V_{ \lambda_i+2u_i ,u_i}  $ for any $ u \in \Gamma_{(v,m)} $. On the other hand, if $ q_i $ is a root of unity for some, recall that $ \dim \Lcell_{ n_i, u_i, \delta_i} = \dim \V_{ n_i,u_i} $ whenever $ n_i-2u_i+1=0 \pmod{ \mathbf{ l}_i}$. Since $ (\lambda_i+2u_i)-2u_i+1=0 \pmod{ \mathbf{ l}_i }$, so $ \Lcell_{ \lambda_i+2u_i, u_i, \delta_i } = \V_{ \lambda_i+2u_i ,u_i}  $. Now, by substituting in equation \eqref{eq5:002}, we obtain $ \dim \Lcell_n ( \lambda ) = \dim \Delta_n ( \lambda ) $, we are done.
\end{proof}

Next we are going to compute the Loewy length and Loewy layers for each cell module of the bubble algebra.

\begin{theorem} \label{thm5:Rad3}
Let $ \Tn $ be the bubble algebra over the complex field and $ \lambda_0+ \lambda_1= n-2v $, $ \lambda_i+t_i+1=0 \pmod{  \mathbf{ l}_i  } $ where $ i=0,1 $ and $ 0 <  t_i < \mathbf{l}_i $, then  
\begin{align*}
 \Lcell_n ( \lambda+2t )  \hookrightarrow \Rad ( \Delta_n ( \lambda  ) ) \twoheadrightarrow  \Lcell_n ( \lambda_0+2t_0, \lambda_1   ) \oplus  \Lcell_n ( \lambda_0 , \lambda_1+2t_1 ) ,
\end{align*}
is an exact sequence,  where $ t=(t_0,t_1) $.  Whenever $ x_0+x_1 > n $, we put $  \Lcell_n ( x_0, x_1 )= \{ 0 \} $  for any $ x_0,x_1 \in \mathbb{N} $.
\end{theorem}
\begin{proof}
Let $ \RV_{ u_i , i } := \RV_{ \lambda_i+2u_i,u_i , \delta_i }$ and $  \V_{ u_i }  :=  \V_{ \lambda_i+2u_i,u_i } $. Define $ W_1,W_2 $ and $ W_{1,2} $ to be 
\[ W_1: =  \sum_{ u \in \Gamma_{(v,2)} } \;  \sum_{ \sigma  \in \widehat{\Sym}_{n,2} }    \sigma (   \,   \RV_{ u_0 ,0 }  \otimes  \V_{u_1  } \; )  , \;\;\;\;\;\; W_2: =  \sum_{ u \in \Gamma_{(v,2)} } \;  \sum_{ \sigma  \in \widehat{\Sym}_{n,2} }    \sigma ( \; \V_{u_0  }  \otimes \RV_{ u_1 ,1 } \; )  , \]
\[ W_{1,2}: =  \sum\limits_{ u \in \Gamma_{(v,2)} } \;  \sum\limits_{ \sigma  \in \widehat{\Sym}_{n,2} }    \sigma ( \; \RV_{ u_0 ,0 }  \otimes  \RV_{u_1,1  } \; ) .\]

Note that $ \Rad ( \Delta_n ( \lambda  ) ) = W_1+W_2 $, see Theorem \ref{thm5:rad}, and $ W_{1,2}=W_1 \cap W_2 $. To prove our theorem we are going to show that $  \Lcell_n ( \lambda+2t ) \cong W_{1,2} $ and $ (W_1+W_2) / W_{1,2} \cong \Lcell_n ( \lambda_0+2t_0, \lambda_1   ) \oplus  \Lcell_n ( \lambda_0 , \lambda_1+2t_1 ) $.

First, we need to show that $ W_1 $ and $ W_2 $ are sub-modules of the module $ \Rad ( \Delta_n ( \lambda  ) ) $. This implies to $ W_{1,2} $ is also sub-module. Let $ x = \sigma ( a_0 \otimes a_1) \in W_1 $ where $ \sigma \in \widehat{\Sym}_{n,2} $, $ a_0 \in \RV_{u_0,0} $ and $ a_1 $ is a link state in $ \V_{u_1} $ for some $ u $, and let $ D=(D_0, D_1) \in \tn $. Since $ \Rad ( \Delta_n ( \lambda  ) ) =W_1+W_2 $, so $ x $ and $ Dx $ is also contained in $  \Rad ( \Delta_n ( \lambda  ) ) $. Without losing the generality, we can assume $ Dx \neq 0 $, then from the graphical visualization we have $ Dx= \zeta \big( ( D'_0a_0)  \otimes (D'_1 a_1) \big) $  for some $ \zeta \in \widehat{\Sym}_{n,2} $ and $ D'_i $ is the diagram $ D_i $ after ignoring the colour, see next example. Note that there exists $ D_1 $ such that $ D'_1 a_1 $ will be not contained $ \RV_{u'_1,1} $ for some $ u'_1 $ since $ \delta_1\neq 0 $(as $ \V_{u'_1} \neq \RV_{u'_1,1} $),  thus $ D'_0a_0 \in \RV_{u'_0,0} $ for some $ u' \in \Gamma_{(v,2)} $, from this we have $ Dx \in W_1 $, this implies to $ W_1 $ is a sub-module. Similarly, $ W_2 $ is a sub-module of $  \Rad ( \Delta_n ( \lambda  ) )$.

From Theorem \ref{thm1:cellhom2}, there is an $ \mathsf{TL}_{\lambda_i+2u_i}( \delta_i)  $-isomorphism $ f_{u_i,i} :  \Lcell_{ u_i , i }  \rightarrow \RV_{ u_i, i } $, where $ \Lcell_{ u_i, i } : = \Lcell_{ \lambda_i+2u_i,u_i , \delta_i }$, since $ \lambda_i+t_i+1=0 \pmod{  \mathbf{ l}_i  } $ and $ 0 <  t_i < \mathbf{l}_i $  for each $ u \in \Gamma_{(v,2)} $. By using one of these isomorphisms we can define a non-zero homomorphism from $  \Lcell_n ( \lambda+2t ) $ to $ W_{1,2} $ as follows: Fix $ u \in \Gamma_{(v,2)} $ and $ f_{u_0,0} $ and $ f_{u_1,1} $. By using $ f_{u_0,0} $ and $ f_{u_1,1} $ we obtain an $ \mathsf{TL}_{\lambda_0+2u_0}(\delta_0) \otimes \mathsf{TL}_{\lambda_1+2u_1} (\delta_1)$-module isomorphism $  f_{u_0,0} \otimes f_{u_1,1}  $ from $ \Lcell_{u_0,0} \otimes \Lcell_{u_1,1} $ to $ \RV_{u_0,0} \otimes \RV_{u_1,1} $. Hence we can define the map $ \Psi :  \Lcell_n ( \lambda+2t ) \rightarrow W_{1,2} $ to be the extension of $  f_{u_0,0} \otimes f_{u_1,1}  $ as $  \Lcell_n ( \lambda+2t ) = \sum_{ w \in \Gamma_{(v,2)} }   \sum_{ \sigma  \in \widehat{\Sym}_{n,2} }    \sigma ( \; \Lcell_{ w_0 ,0 }  \otimes  \Lcell_{w_1,1  } ) $. To prove that $ \Psi $ is a non-zero $ \TN $-homomorphism, it is enough to show that $ \Psi $ is well-defined, i.e. if $ D(a_0 \otimes a_1)=0 $ where $ D \in \tn $ and $ a_i \in \Lcell_{u_i,i} $ for each $ i $, then $ D ( f_{u_0,0} (a_0) \otimes f_{u_1,1}(a_1)) =0 $. From the way we define  $ \RV_{ u_0 ,0 }  \otimes  \RV_{u_1,1  } $ and $ \Lcell_{ u_0 ,0 }  \otimes  \Lcell_{u_1,1  } $, we have $ a_0 \otimes a_1 $ and $  f_{u_0,0} (a_0) \otimes f_{u_1,1}(a_1) $ have the same top $  \underline{\lambda+2u } $, and since the set $ \Top ( D) $ and the arcs in top half of $ D $ don't have any effect on the product, so without losing the generality we can assume that $ D= D_0 \otimes D_1  $ where $D_i \in \mathsf{TL}_{\lambda_i+2u_i} $, but then every thing comes directly from the fact $  f_{u_0,0} \otimes f_{u_1,1}  $ is an $ \mathsf{TL}_{\lambda_0+2u_0}(\delta_0) \otimes \mathsf{TL}_{\lambda_1+2u_1} (\delta_1)$- module isomorphism, hence $ \Psi $ is a non-zero homomorphism. Now as $  \Lcell_n ( \lambda+2t ) $ is simple module and the fact that they have the same dimension, finally by using the first isomorphism theorem we obtain that they are isomorphic.

 Now since $ W_{1,2}=W_1\cap W_2 $, so $ (W_1+W_2)/ W_{1,2} \cong W_1/W_{1,2} \oplus W_2/ W_{1,2}$ . Also 
\begin{align*}
 \frac{W_1 }{W_{1,2}}  = \frac{ \sum\limits_{ u \in \Gamma_{(v,2)} } \;  \sum\limits_{ \sigma  \in \widehat{\Sym}_{n,2} }    \sigma  \big( \,   \RV_{ u_0 ,0 }  \otimes  \V_{u_1  } \big) }{ \sum\limits_{ u \in \Gamma_{(v,2)} } \;  \sum\limits_{ \sigma  \in \widehat{\Sym}_{n,2} }    \sigma  \big( \,   \RV_{ u_0 ,0 }  \otimes  \RV_{u_1,1  } \big) } ,  \cong  \sum\limits_{ u \in \Gamma_{(v,2)} } \;  \sum\limits_{ \sigma  \in \widehat{\Sym}_{n,2} }    \sigma   \,   \RV_{ u_0 ,0 }  \otimes  \Lcell_{ u_1, 1 }.
 \end{align*}
 Now as we did to prove the isomorphism between $ W_{1,2} $ and $\Lcell_n ( \lambda+2t )  $, we can show that $ \sum\limits_{ u \in \Gamma_{(v,2)} } \;  \sum\limits_{ \sigma  \in \widehat{\Sym}_{n,2} }    \sigma   \,   \RV_{ u_0 ,0 }  \otimes  \Lcell_{ u_i, i } \cong  \Lcell_n ( \lambda_0+2t_0, \lambda_1   )  $. Similarly, $ W_2 /W_{1,2}  \cong \Lcell_n ( \lambda_0, \lambda_1+2t_1   )$.
\end{proof}

\begin{example}
Let $ \delta_0=\delta_1=1 $. It is easy to show that \begin{tikzpicture}
          \draw[gray]    (0,0.4) -- (1,0.4);
            \draw (0.75,0)--(0.75, 0.4 );
            \draw [domain=180:360] plot ({0.125*cos(\x)+0.375}, {0.25*sin(\x)+0.4});
            \node[above] at (1.3,0) { $ - $};
\end{tikzpicture}
\begin{tikzpicture}
          \draw[gray]    (0,0.4) -- (1,0.4);
            \draw (0.25,0)--(0.25, 0.4 );
            \draw [domain=180:360] plot ({0.125*cos(\x)+0.625}, {0.25*sin(\x)+0.4});
\end{tikzpicture} is an element in $ \RV_{ 3 ,1, \delta_0 } $, so the element  \begin{tikzpicture}
          \draw[gray]    (0,0.4) -- (1.75,0.4);
            \fill[red] (0.25,0.4)  circle[radius=1.5pt];
            \fill[red] (0.5, 0.4)  circle[radius=1.5pt];
            \fill[blue] (0.75,0.4)  circle[radius=1.5pt];
            \fill[red] (1, 0.4)  circle[radius=1.5pt];
            \fill[blue] (1.25,0.4)  circle[radius=1.5pt];
            \fill[blue] (1.5, 0.4 )  circle[radius=1.5pt];
            \draw[red] (1,0)--(1, 0.4 );
             \draw[blue] (1.5,0)--(1.5, 0.4 );
            \draw [domain=180:360][red] plot ({0.125*cos(\x)+0.375}, {0.25*sin(\x)+0.4});
            \draw [domain=180:360][blue] plot ({0.25*cos(\x)+1}, {0.25*sin(\x)+0.4});
            \node[above] at (2.15,0) { $ - $};
\end{tikzpicture}
\begin{tikzpicture}
          \draw[gray]    (0,0.4) -- (1.75,0.4);
            \fill[red] (0.25,0.4)  circle[radius=1.5pt];
            \fill[red] (0.5, 0.4)  circle[radius=1.5pt];
            \fill[blue] (0.75,0.4)  circle[radius=1.5pt];
            \fill[red] (1, 0.4)  circle[radius=1.5pt];
            \fill[blue] (1.25,0.4)  circle[radius=1.5pt];
            \fill[blue] (1.5, 0.4 )  circle[radius=1.5pt];
            \draw[red] (0.25,0)--(0.25, 0.4 );
             \draw[blue] (1.5,0)--(1.5, 0.4 );
            \draw [domain=180:360][red] plot ({0.25*cos(\x)+0.75}, {0.25*sin(\x)+0.4});
            \draw [domain=180:360][blue] plot ({0.25*cos(\x)+1}, {0.25*sin(\x)+0.4});
\end{tikzpicture} is contained in  $ \Rad ( \Delta_6 (1,1)) $, since it is an element in   $ \sigma \big(  \RV_{ 3 ,1, \delta_0 } \otimes  \V_{ 3 ,1} \big) $ for some $ \sigma  \in \widehat{\Sym_{6,2}} $. Also 
\begin{center}
\begin{tikzpicture}
          \draw[gray]    (0.1,0) rectangle (1.65,0.75);
            \fill[red] (0.25,0)  circle[radius=1.5pt];
            \fill[red] (0.5, 0)  circle[radius=1.5pt];
            \fill[blue] (0.75,0)  circle[radius=1.5pt];
            \fill[red] (1, 0)  circle[radius=1.5pt];
            \fill[blue] (1.25,0)  circle[radius=1.5pt];
            \fill[blue] (1.5, 0 )  circle[radius=1.5pt];
             \fill[red] (0.25,0.75)  circle[radius=1.5pt];
            \fill[blue] (0.5, 0.75)  circle[radius=1.5pt];
            \fill[red] (0.75,0.75)  circle[radius=1.5pt];
            \fill[red] (1,0.75)  circle[radius=1.5pt];
            \fill[red] (1.25,0.75)  circle[radius=1.5pt];
            \fill[red] (1.5, 0.75 )  circle[radius=1.5pt];           
            \draw[red] (1, 0.75)--(0.25, 0 );
            \draw[red] (1.25,0.75)--(0.5, 0 );
            \draw[red] (1.5,0.75)--(1, 0 );
             \draw[blue] (0.5,0.75)--(0.75, 0 );
            \draw [domain=180:360][red] plot ({0.25*cos(\x)+0.5}, {0.25*sin(\x)+0.75});
            \draw [domain=180:0][blue] plot ({0.125*cos(\x)+1.375}, {0.25*sin(\x)+0});
\end{tikzpicture}
\begin{tikzpicture}
            \node[above] at (-0.2,-0.3) { $ \Big( $};
          \draw[gray]    (0,0.5) -- (1.65,0.5);
            \fill[red] (0.25,0.5)  circle[radius=1.5pt];
            \fill[red] (0.5, 0.5)  circle[radius=1.5pt];
            \fill[blue] (0.75,0.5)  circle[radius=1.5pt];
            \fill[red] (1, 0.5)  circle[radius=1.5pt];
            \fill[blue] (1.25,0.5)  circle[radius=1.5pt];
            \fill[blue] (1.5, 0.5 )  circle[radius=1.5pt];
            \draw[red] (1,0)--(1, 0.5 );
             \draw[blue] (1.5,0)--(1.5, 0.5 );
            \draw [domain=180:360][red] plot ({0.125*cos(\x)+0.375}, {0.25*sin(\x)+0.5});
            \draw [domain=180:360][blue] plot ({0.25*cos(\x)+1}, {0.25*sin(\x)+0.5});
            \node[above] at (2,0) { $ - $};
\end{tikzpicture}
\begin{tikzpicture}
          \draw[gray]    (0,0.5) -- (1.65,0.5);
            \fill[red] (0.25,0.5)  circle[radius=1.5pt];
            \fill[red] (0.5, 0.5)  circle[radius=1.5pt];
            \fill[blue] (0.75,0.5)  circle[radius=1.5pt];
            \fill[red] (1, 0.5)  circle[radius=1.5pt];
            \fill[blue] (1.25,0.5)  circle[radius=1.5pt];
            \fill[blue] (1.5, 0.5 )  circle[radius=1.5pt];
            \draw[red] (0.25,0)--(0.25, 0.5 );
             \draw[blue] (1.5,0)--(1.5, 0.5 );
            \draw [domain=180:360][red] plot ({0.25*cos(\x)+0.75}, {0.25*sin(\x)+0.5});
            \draw [domain=180:360][blue] plot ({0.25*cos(\x)+1}, {0.25*sin(\x)+0.5});
            \node[above] at (2,-0.3) { $ \Big)= $};
\end{tikzpicture}
\begin{tikzpicture}
          \draw[gray]    (0.1,0) rectangle (1.65,0.75);
            \fill[red] (0.25,0)  circle[radius=1.5pt];
            \fill[red] (0.5, 0)  circle[radius=1.5pt];
            \fill[red] (0.75,0)  circle[radius=1.5pt];
            \fill[red] (1, 0)  circle[radius=1.5pt];
            \fill[red] (1.25,0)  circle[radius=1.5pt];
            \fill[blue] (1.5, 0 )  circle[radius=1.5pt];
             \fill[red] (0.25,0.75)  circle[radius=1.5pt];
            \fill[blue] (0.5, 0.75)  circle[radius=1.5pt];
            \fill[red] (0.75,0.75)  circle[radius=1.5pt];
            \fill[red] (1, 0.75)  circle[radius=1.5pt];
            \fill[red] (1.25,0.75)  circle[radius=1.5pt];
            \fill[red] (1.5, 0.75 )  circle[radius=1.5pt];           
            \draw[red] (0.25,0.75)--(0.25, 0 );
            \draw[red] (0.75,0.75)--(0.5, 0 );
            \draw[red] (1,0.75)--(0.75, 0 );
            \draw[red] (1.25,0.75)--(1, 0 );
            \draw[red] (1.5,0.75)--(1.25, 0 );
             \draw[blue] (0.5,0.75)--(1.5, 0 );
\end{tikzpicture}
\begin{tikzpicture}
            \node[above] at (-0.2,-0.3) { $ \Big( $};
          \draw[gray]    (0,0.5) -- (1.75,0.5);
            \fill[red] (0.25,0.5)  circle[radius=1.5pt];
            \fill[red] (0.5, 0.5)  circle[radius=1.5pt];
            \fill[red] (0.75,0.5)  circle[radius=1.5pt];
            \fill[red] (1, 0.5)  circle[radius=1.5pt];
            \fill[red] (1.25,0.5)  circle[radius=1.5pt];
            \fill[blue] (1.5, 0.5 )  circle[radius=1.5pt];
            \draw[red] (1.25,0)--(1.25, 0.5 );
             \draw[blue] (1.5,0)--(1.5, 0.5 );
            \draw [domain=180:360][red] plot ({0.125*cos(\x)+0.375}, {0.25*sin(\x)+0.5});
            \draw [domain=180:360][red] plot ({0.125*cos(\x)+0.875}, {0.25*sin(\x)+0.5});
            \node[above] at (2.15,0.1) { $ - $};
\end{tikzpicture}
\begin{tikzpicture}
          \draw[gray]    (0,0.5) -- (1.75,0.5);
            \fill[red] (0.25,0.5)  circle[radius=1.5pt];
            \fill[red] (0.5, 0.5)  circle[radius=1.5pt];
            \fill[red] (0.75,0.5)  circle[radius=1.5pt];
            \fill[red] (1, 0.5)  circle[radius=1.5pt];
            \fill[red] (1.25,0.5)  circle[radius=1.5pt];
            \fill[blue] (1.5, 0.5 )  circle[radius=1.5pt];
            \draw[red] (0.75,0)--(0.75, 0.5 );
             \draw[blue] (1.5,0)--(1.5, 0.5 );
            \draw [domain=180:360][red] plot ({0.125*cos(\x)+0.375}, {0.25*sin(\x)+0.5});
            \draw [domain=180:360][red] plot ({0.125*cos(\x)+1.125}, {0.25*sin(\x)+0.5});
            \node[above] at (2.15,-0.3) { $ \Big),$};
\end{tikzpicture}
\end{center}
note that the element \begin{tikzpicture}
          \draw[gray]    (0,0.4) -- (1.5,0.4);
            \draw(1.25,0)--(1.25, 0.4 );
            \draw [domain=180:360] plot ({0.125*cos(\x)+0.375}, {0.25*sin(\x)+0.4});
            \draw [domain=180:360] plot ({0.125*cos(\x)+0.875}, {0.25*sin(\x)+0.4});
            \node[above] at (2,0) { $ - $};
\end{tikzpicture}
\begin{tikzpicture}
          \draw[gray]    (0,0.4) -- (1.5,0.4);
            \draw (0.75,0)--(0.75, 0.4 );
            \draw [domain=180:360] plot ({0.125*cos(\x)+0.375}, {0.25*sin(\x)+0.4});
            \draw [domain=180:360] plot ({0.125*cos(\x)+1.125}, {0.25*sin(\x)+0.4});
\end{tikzpicture}
 is an element in $ \RV_{ 5 ,2, \delta_0 } $.
\end{example}

\begin{example}
Let $ \breve{\delta}= ( 0, \sqrt{2} )  $, then $ \mathbf{ l}_0=2 $ and $ \mathbf{ l}_1=4 $ and the critical lines are $ \lambda_0=1,3,5, \dots $ and $ \lambda_1=3,7, \dots $ which are represented by coloured lines in figure \ref{fig5:59}. Also the arrows in the figure represent non-zero homomorphisms between the cell modules that are indexed by the nodes in the figure. Two nodes will be in the same block if and only if there is an arrow between them. Then decomposition matrix of the algebra $ \mathbb{T}_{6,2} ( 0, \sqrt{2} ) $  is 
 \[  \begin{pmatrix}
1 &  1 & 1 & 0 & 0  \\
0 &  1 & 0 & 1 & 0   \\
0 &  0 & 1 & 0 & 0  \\
0 &  0 & 0 & 1 & 1 \\
0 &  0 & 0 & 0 & 1 
\end{pmatrix}  \bigoplus   \begin{pmatrix}
1 &  1    \\
0 &  1  
\end{pmatrix}  \bigoplus  \begin{pmatrix}
1 &  1 & 1 & 0 & 1 \\
0 &  1 & 0& 1 & 1   \\
0 &  0 & 1 & 0& 1  \\
0 &  0 & 0 & 1 & 0 \\
0 &  0 & 0 & 0 & 1 
\end{pmatrix}  \bigoplus \bigoplus\limits^{4} (1), \]
we order the basis as following $ \{ (0,0),(2,0), (0,6), (4,0), (6,0), (1,1),(1,5), (0,2), (2,2),$  $ (0,4), (4,2),(2,4), (3,1), (1,3), (5,1), (3,3) \} $. Then the Cartan matrix of $ \mathbb{T}_{6,2} ( \breve{ \delta } ) $ is 
 \[  \begin{pmatrix}
1 &  1 & 1 & 0 & 0  \\
1 &  2 & 1& 1 & 0   \\
1 &  1 & 2 & 0 & 0  \\
0 &  1 & 0 & 2 & 1 \\
0 &  0 & 0 & 1 & 2 
\end{pmatrix}  \bigoplus   \begin{pmatrix}
1 &  1    \\
1 &  2  
\end{pmatrix}  \bigoplus  \begin{pmatrix}
1 &  1 & 1 & 0 & 1 \\
1 &  2 & 1& 1 & 2   \\
1 &  1 & 2 & 0& 2 \\
0 &  1 & 0 & 2 & 1 \\
1 &  2 & 2 & 1 & 4 
\end{pmatrix}  \bigoplus \bigoplus\limits^{4} (1). \]

\begin{figure}[!h] 
  \centering
    \includegraphics[width=120mm,scale=1]{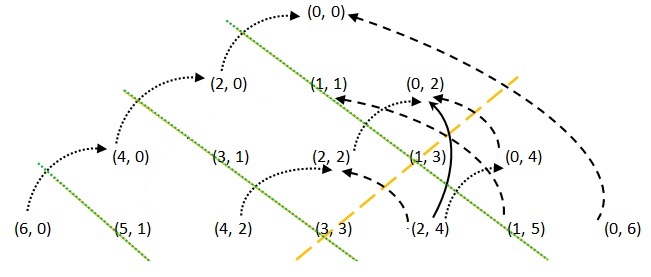}
    \caption{The Bratteli diagram of $ \mathbb{T}_{6,2} ( \breve{ \delta } )$ when $ \mathbf{ l}_0=2$ and $ \mathbf{ l}_1=4 $.}
    \label{fig5:59}
\end{figure}
 
\end{example}

 \hspace{9pt} Next theorem is a generalization of last theorem in the case $ m>2 $ with several parameters are roots of unity. 

\begin{theorem} \label{thm5:general}
Let $ \TM ( \breve{\delta} ) $ be the bubble algebra over the complex field and $ \lambda \in \Gamma_{(n-2v,m)} $, $ 0 \leq  s<m $. For each $ i>s $, suppose either $ q_i $ is not a root of unity or $ \lambda_i+1=0 \pmod{  \mathbf{ l}_i } $ when $ q_i $ is a root of unity, and  for each $ j \leq s $ we have $ \lambda_j+t_j+1=0 \pmod{  \mathbf{ l}_j } $ and $ 0 <  t_j < \mathbf{l}_j $. Then the length of the radical series of $  \Delta_n ( \lambda ) $ is less than or equal to $ s+1 $, and the radical layers are
\[ \Rad^{k} ( \Delta_n ( \lambda )) / \Rad^{k+1} ( \Delta_n ( \lambda )) \cong \bigoplus\limits_{ \lambda' \in \Xi_k } \Lcell_n ( \lambda' ) , \]
where $ \Xi_k = \{ \lambda' | \text{ there are exactly } k \; \text{values of } j \text{ where } 0 \leq j \leq s \text{ such that } \lambda'_j=\lambda_j+2t_j \text{ and for the other values we have } \lambda'_i = \lambda_i \} $ and $ 0 \leq k \leq s+1 $. We put $ \Lcell_n ( \lambda' ) = \{ 0 \} $ whenever $ \sum \lambda'_i > n $.
\end{theorem}
\begin{proof}
From Theorem \ref{thm3:rad7}, we have 
\begin{align*} 
 \Rad ( \Delta_n( \lambda )  ) =    \sum\limits_{ u \in \Gamma_{(v,m)} } \;  \sum\limits_{ \sigma  \in \widehat{\Sym}_{n,m} }  \sum\limits_{i=0}^{s}  \sigma  \big( \,\V_{ 0 } \otimes \cdots \otimes \V_{i-1} \otimes  \RV_{ i }  \otimes  \V_{ i+1 } \otimes \cdots \otimes \V_{ m-1} \, \big) ,
\end{align*}
where $ \RV_{ i } := \RV_{ \lambda_i+2u_i,u_i , \delta_i }$ and $  \V_{ i }  :=  \V_{ \lambda_i+2u_i,u_i } $, since $ \RV_{i}= \{ 0 \} $ for each $ i>s $. Define $ W_i $  to be 
\[ W_i = \sum\limits_{ u \in \Gamma_{(v,m)} } \;  \sum\limits_{ \sigma  \in \widehat{\Sym}_{n,m} }  \sigma  \big( \,\V_{ 0 } \otimes \cdots \otimes \V_{i-1} \otimes  \RV_{ i }  \otimes  \V_{ i+1 } \otimes \cdots \otimes \V_{ m-1}  \big) , \]
 where $ 0 \leq i \leq s $. Note that $ W_i $ is a sub-module of $  \Rad ( \Delta_n( \lambda )  ) $ for each $ i $, the proof is similar to the one in Theorem \ref{thm5:Rad3}. Also define the modules $ W_{i_1, \ldots, i_k} $ and $ W^{k} $, where $ 0 \leq i_h \leq s $ for each $ h $ and $ i_h \neq i_{h'} $ for each $ h \neq h' $ where $ k=1, \dots , s+1 $, to be 
\begin{align*}
 W_{i_1, \ldots, i_k}  = \bigcap\limits_{h=1}^{k} W_{i_h},  \;\;\;\;\;\;\;\;\;\;  W^{k} = \sum\limits_{(i_1, \ldots, i_k)} W_{i_1, \ldots, i_k}.
 \end{align*}
From their definitions, it is clear that $ \sum\limits_{ i_k} W_{i_1, \ldots, i_k } \subseteq W_{i_1, \ldots, i_{k-1}}$, thus $ W^{k} \subseteq W^{k-1} $.

 We are going to prove that $  \Rad^k ( \Delta_n( \lambda )  ) = W^k $, by using induction where it is clear that $  \Rad ( \Delta_n( \lambda )  ) = W^1 $ and $ W^{k+1} \subseteq W^{k} $, we only need to show $ W^k/ W^{k+1} \cong \bigoplus\limits_{ \lambda' \in \Xi_k } \Lcell_n ( \lambda' ) $:
\begin{align*}
\frac{W^k}{W^{k+1} } = \frac{\sum\limits_{(i_1, \ldots, i_k)} W_{i_1, \ldots, i_k}}{\sum\limits_{(j_1, \ldots, j_{k+1} )} W_{j_1, \ldots, j_{k+1}}}  \cong \bigoplus\limits_{(i_1, \ldots, i_k)} \frac{W_{i_1, \ldots, i_k}}{ \sum\limits_{i_{k+1}} W_{i_1, \ldots, i_k, i_{k+1}} }.
\end{align*}
Without loss generality, we will just compute $ W_{0, \ldots, k-1}/( \sum\limits_{i=k}^s W_{0, \ldots,k-1, i} ) $ which equals 
\begin{align*}
\frac{\sum\limits_{ u } \;  \sum\limits_{ \sigma }  \sigma  \big( \,\bigotimes\limits_{j=0}^{k-1} \RV_{ j } \otimes \bigotimes\limits_{h=k}^{m-1} \V_{ h} \, \big) }{ \sum\limits_{i=k}^s \sum\limits_{ u  } \;  \sum\limits_{ \sigma }  \sigma  \big( \, \bigotimes\limits_{j=0}^{k-1} \RV_{ j }  \otimes  \V_{k }  \otimes \cdots \otimes \V_{i-1} \otimes \RV_i  \otimes  \V_{i+1 }  \otimes \cdots \otimes \V_{ m-1} \, \big) },
\end{align*}
this module is isomorphic to 
\[ Z:= \sum_{ u \in \Gamma_{(v,m)} } \;  \sum_{ \sigma \in \widehat{\Sym}_{n,m} }  \sigma  \big( \,\bigotimes\limits_{j=0}^{k-1} \RV_{ j } \otimes \bigotimes\limits_{l=k}^{s} \Lcell_{ h} \otimes \bigotimes\limits_{l=s+1}^{m-1} \V_{ l}  \, \big),\]
 where $ \Lcell_{ i } := \Lcell_{ \lambda_i+2u_i,u_i , \delta_i }$. Since $ \V_l \cong \Lcell_l $ for each $ l>s $, and from Theorem \ref{thm1:cellhom2} there is a non-zero homomorphism from $\Lcell_{ \lambda_i+2t_i+2u_i,u_i , \delta_i }$ to $ \RV_{ \lambda_i+2u_i,u_i , \delta_i } $ for each $ i >k $. Hence we can define a non-zero homomorphism from $ \Lcell_n ( \lambda') $ to $ Z $, also  we can show that they have the same dimension, so they are isomorphic by using the first isomorphism theorem, where $ \lambda'= ( \lambda_0+2t_0, \dots , \lambda_{k-1}+2t_{k-1}, \lambda_k, \dots , \lambda_{m-1}) $. It is clear that $ \lambda' \in \Xi_k  $ and by taking all the possibilities of the tuple $ (i_1, \dots , i_k) $ we will obtain all the elements in the set $ \Xi_k  $, we are done. 
\end{proof}

\bibliographystyle{plain}

\end{document}